\documentclass[12pt,reqno]{amsart}
\usepackage{hyperref}
\usepackage{amsmath}
 \usepackage{bm,amsthm,amsfonts,latexsym,amssymb,commath} 

\usepackage{graphicx,float,multicol}

\newtheorem{lemma}{\bf{Lemma} }[section]
\newtheorem{proposition}{\bf{Proposition}}[section]
\newtheorem{theorem}{\bf{Theorem}}[section]
\newtheorem{remark}{\sc{Remark} }[section]
\newtheorem{definition}{\sc{Definition} }[section]
\newtheorem{corollary}{\bf{Corollary} }[section]

\usepackage{siunitx}

\begin{document}

 \title[Compressible NSF flows at steady-state]{Compressible Navier--Stokes--Fourier flows at steady-state}

\author{Luisa Consiglieri}
\thanks{Dedicated to my coauthor and beloved father Victor Consiglieri.}
\address{Luisa Consiglieri, Independent Researcher Professor, European Union}
\urladdr{\href{http://sites.google.com/site/luisaconsiglieri}{http://sites.google.com/site/luisaconsiglieri}}

\begin{abstract} 
The heat conducting compressible viscous flows are governed by the Navier--Stokes--Fourier (NSF) system.
In this paper, we study the NSF system accomplished by 
the Newton law of cooling for the heat transfer at the boundary. On one part of the boundary, we consider
the Navier slip boundary condition, while in the remaining part the inlet and outlet occur.
The existence of a weak solution is proved via a new fixed point argument.  
With this new approach, the weak solvability is possible in Lipschitz domains, by making recourse to \(L^q\)-Neumann problems with \(q>n\).
Thus,  standard existence results can be applied to auxiliary problems and the claim follows by compactness techniques.
Quantitative estimates are established.
\end{abstract}


\maketitle

\section{Introduction}

The heat conductive flows are described by a
coupled system consisting of the equations of continuity, motion and energy.
The study of compressible flows depends on the knowledge of solving the continuity equation, 
because this equation has its shortcomings.
 We refer to \cite{bveiga87} for the existence of stationary solutions
if the transport coefficients are, at least, of class \(W^{2,p}(\Omega)\) with \(p>n\).

Several works deal with barotropic flows, where   the pressure is a function of the density only.
 To cover the physical point of view, namely,  the adiabatic exponent \(\gamma=5/3\) for the monoatomic gases
or \(\gamma=7/5\) for the diatomic gases
at ordinary temperature \SIrange{150}{600}{\kelvin},
the imposed assumption on the pressure has being studied in function of the adiabatic exponent
\(\gamma\). To deal with this,
 the renormalized bounded energy weak solutions, in the context of the theory introduced by P.L. Lions \cite{plions},
 are proved for \(\gamma \geq 5/3\) if \(n=3\). Since then the adiabatic exponent is becoming
 realistic. In \cite{frehse-w}, the renormalized bounded energy weak solutions are proved
  under the assumption that the adiabatic exponent satisfies \(\gamma> 4/3\).
 We refer the existence of renormalized weak solutions
  for \(\gamma> (3+\sqrt{41})/8\) to \cite{brezina}, for the flows powered by volume potential forces
 in a rectangular domain with periodic boundary conditions,
 and recently, for \(\gamma>1\) to \cite{plotni-w}, in a bounded domain with  no-slip boundary condition.
For a general case, the existence of a fixed point to the Navier--Stokes system is applied in  \cite{valli}
by using the Schauder theorem under smallness of the \(H^3\)-norm for the velocity field if
providing the system by smooth coefficients. The higher order derivatives are essential in establishing
the estimate of div\,\(\mathbf{u}\).
We remind that a fluid that flows at low velocity is described by the Stokes equations
 and not by the Navier--Stokes equations.

Nonisothermal steady state studies are well known and there exists a vast literature under the Dirichlet condition, for instance,
 on optimal control of low Mach number \cite{imanu}
and on uniqueness \cite{padula} and the literature cited therein.
 The better regularity of solutions by introducing the effective viscous flux 
 \(G=p-(2\mu +\lambda)\mbox{div}\,\mathbf{u}\) is only possible under constant viscosities \(\mu\) and \(\lambda\)
(see \cite{frehse-w},  and the references therein). With this assumption, the authors in \cite{mucha}
prove the existence of weak solutions by replacing,  in NSF system, the energy equation by
the total energy equation. This new system has the particularity of adding the equations,
the pressure and the dissipation disappear, in the establishment of the crucial estimates.

Here, we consider the transport coefficients as temperature and spatial dependent. 
The behavior of the transport coefficients do not allow  standard techniques \cite{sarka2009} as, for instance,
the use of either the above \(G\) or the inverse of the Stokes operator.

The inhomogeneous boundary value problems are, in contrast, less common.  We refer to \cite{plotni} to the
existence of continuous strong solutions to NSF problem under the assumptions
that the Reynolds number  and the inverse viscosity ratio are small and the Mach number Ma\,\(\ll 1\).

The study of the NSF system that the source/sink is the heat transfer at the boundary,
 which is given by the Newton law of cooling, 
 can be applicable to the physical situations such that come from  biomedical engineering (as, for instance,
thermal ablation for the treatment of thyroid nodules \cite{chung,radz}) as well as geological engineering
(as, for instance, the natural gas flow in wells at the region that a single phase occurs).

\textit{A priori} estimates are the core in a fixed point argument.
 However, they are usually deduced from  the boundedness propriety of the operators. 
Then, there exist a universal constant that is abstract, that is, it does not reflect the data dependency.
 To fill this gap, additional attention is payed in the determination of quantitative estimates  in which the dependence
 on the data is explicit.

The outline of this paper is as follows.
Next section is concerned for modeling of the problem under study and the description of the model itself.
Section \ref{smain} is devoted to the mathematical framework, the establishment of the data assumptions,
and the statement of the main theorems.
In Section \ref{strat}, we delineate the fixed point argument.
The following sections (Sections \ref{sZO}, \ref{sdens} and \ref{sSOLA}) concentrate on the wellposedness of three auxiliary problems, namely
a Dirichlet--Navier problem for the velocity field, a inlet/outlet problem for the density scalar and a Dirichlet--Robin problem for the temperature.
The remaining sections (Sections \ref{smain1} and \ref{smain2}) are devoted to the proofs of the main theorems,
respectively, Theorems \ref{main} and \ref{main2}.

\section{Statement of the problem}
\label{stt} 

Let \( \Omega\)  be a bounded domain (connected open set) of  \( \mathbb R^n\), \( (n= 2,3)\),  with Lipschitz boundary.
The boundary \(\partial\Omega\) consists of three pairwise disjoint relatively open \( (n-1)\)-dimensional submanifolds,
 \( \Gamma_\mathrm{in}\), \(\Gamma_\mathrm{out}\) and \( \Gamma\),  with positive Lebesgue measures, 
 whose verify 
\[
\mathrm{cl}(\Gamma_\mathrm{in}) \cup \mathrm{cl}(\Gamma_\mathrm{out}) \cup
\mathrm{cl}(\Gamma)=\partial\Omega,
\]
where cl stands for the set closure.

The heat conducting fluid at steady-state is governed by the Navier--Stokes--Fourier equations 
\begin{align}
 \nabla\cdot(\rho\mathbf{u})&= 0\label{mass}\\
\rho (\mathbf{u}\cdot\nabla)\mathbf{u} -\nabla\cdot\sigma &=\rho  \mathbf{g}\label{motion}\\
\label{heateqs}
\rho \mathbf{u}\cdot\nabla e-\nabla\cdot (k(\theta)\nabla \theta) &=\sigma:D\mathbf{u}  \mbox{ in } \Omega .
\end{align}
Here, the unknown functions are the density \(\rho\), the  velocity field \(\mathbf{u}\), and the 
 specific internal energy  \( e\). 
 We denote \( \zeta:\varsigma=\zeta_{ij}\varsigma_{ij}\)  taking into account the convention on
implicit summation over repeated indices.
The gravitational force \(\mathbf{g}\) and the  dissipation  \( \sigma:D\mathbf{u}\) are negligible.
 Notice that the neglecting the external force fields does not imply that the fluid is at rest. 
Indeed, the fluid flow is driven both by inlet and outlet flows and  by heat transfer on the boundary.

 In the case of ideal gases, the  specific internal energy  \( e\)
  is related with the absolute temperature \(\theta\) by the linear relationship \( e = c_v \theta,\)
where \( c_v\) denotes the specific heat capacity of the fluid at constant volume.
Thus, the energy equation \eqref{heateqs} can be written in terms of the temperature.
Assuming that  the thermal conductivity  \( k\) is  a function  dependent on both temperature and space variable, 
the smoothness of the temperature depends on this coefficient.

The Cauchy stress tensor \(\sigma\), which is temperature dependent, obeys the constitutive law 
\begin{equation}\label{diff}
\sigma =-p \mathsf{I}+\mu(\theta)D\mathbf{u}+\lambda(\theta) {\rm tr}(D\mathbf{u}) \mathsf{I},
\quad{\rm tr}(D\mathbf{u})=\mathsf{I}:D\mathbf{u}= \nabla\cdot\mathbf{u},
\end{equation}
where \(\mathsf{I}\) denotes the identity (\(n\times n\))-matrix, 
 \( D=(\nabla+\nabla^T)/2\) the symmetric gradient,
and \( \mu\) and \( \lambda\) are the viscosity coefficients  in 
accordance with the second law of thermodynamics
\begin{equation}\label{mu}
\mu(\theta)>0,\quad \nu(\theta):=\lambda(\theta)+\mu(\theta)/n\geq 0,
\end{equation} 
with \(\nu\) denoting the bulk (or volume) viscosity  and \(\mu/2\) being the shear (or dynamic) viscosity.

 The  pressure  \( p\) in the case of ideal gases obeys to the Boyle--Marriotte law
 \begin{equation}\label{boyle}
 p= R_\mathrm{specific}\rho\theta 
\end{equation}
where \( R_\mathrm{specific} =R/M\) is the specific gas constant, with 
 \(R= \SI{8.314}{\joule\per\mole\per\kelvin}\) being  the gas constant and \(M\) denoting the molar mass. 
 
 To understand the range of values we are talking to about, we exemplify some well known values for the dry air.
 For the air (assumed to be at the atmospheric pressure \(p=\SI{101,325}{\kilo\pascal}\)), 
 the molar mass of dry air is \(M=\SI{28.96}{\kilogram\per\kilo\mole}\) at temperature \(\theta=\SI{298.15}{\kelvin}\)
 (\(=\SI{25}{\celsius}\)), then
 the density \(\rho =\SI{1.184}{\kilogram\per\cubic\metre}\). Thus, we have
\(R_\mathrm{specific}=\SI{287}{\joule\per\kilogram\per\kelvin}\).
The dry air  can be assumed as diatomic,  then \(c_v =5R/2\).
The dynamic viscosity \(\mu/2=\SI{0.018}{\milli\pascal\second}\) and the bulk viscosity  \(\nu = 0.8 \mu/2\) \cite{gu-ubachs}.
Similar values are known for O\(_2\)  (see Table 1).
 \begin{table}[h]\label{table1}
\caption{Parameters at the atmospheric pressure \cite{kmn,lae}}
 \begin{tabular}{|c|c|c|c|c|}
\hline 
\(\theta\) &  \(\mu/2\) (Air) & \(\mu/2\) (O\(_2\)) &  \(k\)   (Air) 
&  \(k\) (O\(_2\))  \\ 
\([\si{\kelvin}]\)
&  \( [10^{-5}\si{\pascal\second}]\)  &  \( [10^{-5}\si{\pascal\second}]\)
&  \( [10^{-2}\si{\watt\per\metre\per\kelvin}]\)& \( [10^{-2}\si{\watt\per\metre\per\kelvin}]\)
 \\
\hline 
100 &0.7 & 0.8 &0.9  & 1.0 \\  
200 &1.3 & 1.5& 1.8 & 1.8 \\ 
300 &1.9& 2.0 &2.6 &  2.7  \\
500 & 2.7 & 3.0 & 4.0 & 4.3 \\ 
800 &3.7 & 4.2 & 5.7 & 6.6 \\  
1000 &4.3 & 4.9 & 6.8 & 8.0 \\ 
\hline 
\end{tabular} 
 \end{table}
 
The triple point of the air is reached at temperature of \SI{59.75}{\kelvin}  (\(=-\SI{213.4}{\celsius}\))
and a correlated pressure (which value varies from author to author because how it is assumed the air composition).
Thus, a minimum temperature \(\theta_0\) is admissible.
 The values for the velocity, however, range from  that the flow has Reynolds number Re\,\(\ll 1\),
in which case is described by the Stokes equation, until the flow behaves in the turbulent regime of Re\,\(\geq 10^6\).
This means Re\,\(> 6.5\times 10^4 v L\), with \(v\) standing for an average velocity and \(L\) the maximum length
of the cross-section of the domain, in the above conditions. 

We notice that, in this work, we only assume as constant the specific heat capacity.
This assumption is essential to leave the thermal conductivity as space variable dependent, by replacing the specific
internal energy by the temperature  as an unknown to seek.
We leave all the remaining coefficients dependent on the temperature  (see, for instance, Table 1) and on the space variable.

On the Dirichlet boundary  \(\Gamma_{D}=
\mathrm{int}(\mathrm{cl}(\Gamma_\mathrm{in})\cup \mathrm{cl}(\Gamma_\mathrm{out}))\),
 we assume inhomogeneous Dirichlet boundary condition 
\begin{equation}\label{bdD}
\rho=\rho_\infty \quad\mbox{and}\quad \mathbf{u}=\mathbf{u}_{D}.
\end{equation}
This represents both the inflow (\(u_\mathrm{in}:=\mathbf{u}_{D}\cdot\mathbf{n} <0\))
 and outflow (\(u_\mathrm{out}:=\mathbf{u}_{D}\cdot\mathbf{n} >0\)).

On the remaining  boundary \(\Gamma\),  the fluid do not penetrate the solid wall,
and it obeys the Navier slip boundary condition
\begin{equation}\label{ud1}
u_N:=\mathbf{u}\cdot\mathbf{n}=0,\qquad \tau_T=-\gamma (\theta)\mathbf{u}_T,
\end{equation}
where \( \mathbf{n}\) stands for the unit outward vector to  \( \Gamma\),
\(u_N,\mathbf{u}_T\) are the normal and  tangential components
of the velocity vector, respectively,  \(\tau_T=\tau\cdot \mathbf{n}-\tau_N\mathbf{n}\) and \(\tau_N =( \tau\cdot \mathbf{n})\cdot\mathbf{n}\)
 are the  tangential and normal components of the deviator stress tensor \(\tau=\sigma +p\mathsf{I}\), respectively,   and
\(\gamma\) denotes the friction coefficient.

For the heat transfer conditions, it is admissible to assume prescribed temperature in the inlet, that is,
we consider the Dirichlet condition 
\begin{equation}
\theta=\theta_\mathrm{in}\ \mbox{ on }\Gamma_\mathrm{in}.\label{tin}
\end{equation}
For the sake of simplicity, we assume \(\theta_\mathrm{in}\) as a positive constant.
Alternatively, we might assume that \(\theta_\mathrm{in}\) may be extended to a function \(\tilde\theta_\mathrm{in}\in H^1(\Omega).\)

On the boundary \(\Gamma_N=\partial\Omega\setminus \mathrm{cl}(\Gamma_\mathrm{in})\), we assume the Newton law of cooling 
 \begin{equation}\label{newton}
k(\theta)\nabla \theta\cdot\mathbf{n}+h_c  (\theta)  (\theta-\theta_\mathrm{e})=0,
 \end{equation}
 where \( h_c\) denotes the heat transfer coefficient and \( \theta_\mathrm{e}\) represents a given
 (eventually nonconstant) external temperature. This condition is mathematically known as the Robin condition.
The heat source/sink is completely driven from the boundary and we denote
\[
\theta_0 =\left\lbrace
\begin{array}{ll}
\theta_\mathrm{in} & \mbox{on }\Gamma_\mathrm{in}\\
\theta_\mathrm{e} = & \left\lbrace
\begin{array}{l}
\theta_\mathrm{w}  \mbox{ on }\Gamma\\
\theta_\mathrm{out}  \mbox{ on }\Gamma_\mathrm{out}.
\end{array} \right.
\end{array}
\right.
\]

\section{Main Results}
\label{smain}

We assume that \( \Omega \subset\mathbb{R}^n\) is a bounded domain with its boundary \( \partial\Omega\in C^{0,1}\).
The standard notation of Lebesgue and Sobolev spaces is used. 
Let us define the Hilbert spaces
\begin{align*} 
{H}_{\mathrm{in}}^{1}(\Omega)&:=\{ v\in  {H}^{1}(\Omega):\ v=0\mbox{ on }\Gamma_\mathrm{in}\};\\
\mathbf{V}&:=\{\mathbf{v}\in \mathbf{H}^{1}(\Omega):\ 
\mathbf{v}=\mathbf{0}\mbox{ on }\Gamma_D,\ \mathbf{v}\cdot\mathbf{n}=0\mbox{ on }\Gamma\},
\end{align*}
 endowed with the norms, respectively,
\begin{align*}
\|v\|_{1,2,\Omega}&=\left( \|\nabla v\|_{2,\Omega}^2+ \|v\|_{2,\Gamma}^2 \right)^{1/2} ;\\
\|\mathbf{v}\|_\mathbf{V}&= \left(\| D\mathbf{v}\|_{2,\Omega}^2+ \|\mathbf{v}\|_{2,\Gamma}^2  \right)^{1/2} .
\end{align*}
The meaning of the condition   \(\mathbf{v}\cdot\mathbf{n}=0\) on  \(\Gamma\) should be understood as
\[
\langle \mathbf{v}\cdot\mathbf{n} ,v \rangle_{\Gamma}=0,\quad\forall v\in { H}^{1/2}_{00}(\Gamma)=
 \{v\in H^{1/2}(\partial \Omega):\ v=0\quad\mbox{on } \Gamma_D \},
\] 
where the symbol  \(\langle\cdot,\cdot\rangle_\Gamma\) stands for
the duality pairing  \(\langle\cdot,\cdot\rangle_{Y'\times Y}\), where \(Y={ H}^{1/2}_{00}(\Gamma)\).

 \begin{definition}[NSF problem]\label{mainbj}
 We say that the triplet \((\rho,\mathbf{u}, \theta)\) is  a weak solution to the NSF problem if
 it satisfies  the integral identities
 \begin{align}\label{rhow}
\int_\Omega \rho \mathbf{u}\cdot\nabla v\dif{x}=
\int_{\Gamma_D}\rho_\infty\mathbf{u}_D\cdot\mathbf{n} v \dif{s}  ,\quad\forall v\in  W^{1,q'}(\Omega); \\
\int_{\Omega}\rho (\mathbf{u}\cdot\nabla )\mathbf{u}\cdot\mathbf{v}\dif{x} 
+\int_{\Omega}\mu(\theta)D\mathbf{u}:D\mathbf{v} \dif{x} 
+\int_{\Omega} \lambda(\theta)\nabla\cdot\mathbf{u}\nabla\cdot\mathbf{v} \dif{x}  +\nonumber\\
+\int_{\Gamma}\gamma(\theta)\mathbf{u}_T\cdot \mathbf{v}_T\dif{s} 
 =\int_{\Omega}  p\nabla \cdot\mathbf{v}\dif{x} ,\quad\forall \mathbf{v}\in \mathbf{V}; \qquad
\label{motionw}\\
c_v \int_{\Omega} \rho \mathbf{u}\cdot\nabla \theta v \dif{x} + \int_\Omega k(\theta)\nabla \theta\cdot\nabla v \dif{x} 
+\int_{\Gamma_N} h_c(\theta) \theta v\dif{s} = \nonumber\\ = \int_{\Gamma_N} h(\theta) v \dif{s} 
 ,\quad\forall v\in H^{1}_\mathrm{in}(\Omega),\label{heatw}
 \end{align} 
 subject to \eqref{boyle}, \eqref{bdD} and \eqref{tin}.
Here, \(q'\) stands for the conjugate exponent of \(q\), \textit{i.e.} \(1/q'+1/q=1\), and  \(h=h_c\theta_\mathrm{e}\).
 \end{definition}

\begin{remark}
 The variational formulations  \eqref{rhow}-\eqref{heatw} are standardly derived from the NSF system \eqref{mass}-\eqref{heateqs} by
 the Green formula. We point out that the general formula
\[
\langle \rho (\mathbf{u}\cdot\nabla)  \mathbf{u} , \mathbf{v}\rangle =
\langle \tau_T,\mathbf{v} \rangle_\Gamma -\langle  p , \mathbf{v}\cdot \mathbf{n}\rangle_\Gamma
-\int_{\Omega}\sigma :D\mathbf{v}\dif{x} 
 \]
holds for any \( \mathbf{v}\in \mathbf{V}\), under \(\nabla\cdot\sigma \in \mathbf{V}'\) \cite{lap2011}.
\end{remark}
 
The following assertions on the physical parameters appearing in the equations are
assumed:
\begin{description}
\item[(H1)] The viscosities \(\mu\) and \(\lambda\) are Carath\'eodory functions from \(\Omega\times\mathbb{R}\)
into \(\mathbb{R}\) such that
\begin{align}\label{mu1}
\exists \mu_\# >0: &\ \mu(x,e) \geq \mu_\#>0  ; \\  \label{mu2}
\exists \mu^\# >0: &\ \mu(x,e) \leq \mu^\# ; \\
\exists \lambda^\# >0: &\ |\lambda(x,e)| \leq  \lambda^\# ,
\label{nu3}
\end{align}
for a.e. \(x\in\Omega\) and  for all \( e\in\mathbb{R}\).

\item[(H2)] The thermal conductivity \(k\) is a  Carath\'eodory  function from \(\Omega\times\mathbb{R}\)
into \(\mathbb{R}\) such that
\begin{equation}\label{defchi}
\exists k^\#, k_\# >0 :  \quad k_\#\leq k(x,e)\leq k^\#,
\end{equation}
for a.e. \(x\in\Omega\) and  for all \( e\in\mathbb{R}\).

\item[(H3)] The friction coefficient \(\gamma\) is a continuous function from \(\mathbb{R}\)
into \(\mathbb{R}\) such that
\begin{equation}\label{g1}
\exists  \gamma^\#,  \gamma_\# >0 :  \quad \gamma_\#\leq \gamma(e)\leq \gamma^\#,\quad \forall e\in \mathbb{R}.
\end{equation}

\item[(H4)] The heat transfer coefficient \(h_c\) is a  Carath\'eodory  function from \(\Gamma_N\times\mathbb{R}\)
into \(\mathbb{R}\) such that
\begin{align}\label{defhm}
\exists h^\# >0: &\ h_c(e)\leq h^\#\mbox{ a.e. on }\Gamma_N;\\
\exists h_\#>0:& \ h_c(e)\geq h_\# \mbox{ a.e. on }\Gamma; \label{h1}\\
 &\ h_c(e)\geq 0\mbox{ a.e. on }\Gamma_\mathrm{out},\label{hout}
\end{align}
for all \(e\in \mathbb{R}\). Moreover, \(h=\theta_\mathrm{e}h_c\) with the function 
\(\theta_\mathrm{e}\in L^\infty (\Gamma_N)\).

\item[(H5)] The boundary term \(\rho_\infty \mathbf{u}_{D}\cdot\mathbf{n}\in L^q (\Gamma_{D})\), for some \(q>n\),
satisfy  the compatibility condition
\begin{equation}\label{cc}
\int_{\Gamma_{D}}\rho_\infty  \mathbf{u}_{D}\cdot\mathbf{n} \dif{s}=0.
\end{equation}
There exists \(\widetilde{\mathbf{u}}_D\in  \mathbf{H}^{1}(\Omega)\) such that its trace   
\(\widetilde{\mathbf{u}}_D= \mathbf{u}_{D}\) on \(\Gamma_D\) and the normal component of trace vanishes  on \(\Gamma\).
 Indeed, the trace operator has a continuous right inverse operator, 
and in particular it is surjective from \( \mathbf{W}^{1,q}(\Omega)\) onto \( \mathbf{W}^{1-1/q,q}(\partial\Omega)\).
\end{description}

\begin{remark}\label{rsob}
We denote by \(p^*=pn/(n-p)\) the critical Sobolev exponent related to the embedding 
\(W^{1,p}(\Omega)\hookrightarrow L^{p^*}(\Omega)\), if \(p<n\). 
For the sake of simplicity, we also denote by \(p^*\) any real value greater than one, if \(p=n\).
The Rellich--Kondrachov embedding stands for any exponent between \(1\) and  the critical Sobolev exponent \(p^*\).
Notice that the Morrey embedding \(W^{1,q}(\Omega)\hookrightarrow C^{0,1-n/q}(\Omega)\) holds for \(q>n\).
\end{remark}

\begin{remark}\label{meaningful}
All terms are meaningful in the integral identities \eqref{rhow}-\eqref{heatw}. 
The nonlinear terms, the convective term in \eqref{motionw} and the advective term in \eqref{heatw},
are justified in Lemma \ref{lemb}, with \(\mathbf{m}=\rho\mathbf{u}\in \mathbf{L}^{q}(\Omega)\), \(q>n\),  \textit{i.e.}
\(\rho\in L^r(\Omega)\) and \(\mathbf{u}\in \mathbf{H}^1(\Omega)\), with
\begin{equation}\label{defr}
\frac{1}{q} = \frac{1}{r}+\frac{1}{p} \quad\mbox{ if } r = \frac{2p}{p-4} >\frac{2n}{4-n} ,\ 4<p<2^*\ (n=2,3).
\end{equation}
Observe that \(r> 2n/(4-n)\) follows from \(1/n> 1/q = 1/r +1/p >1/r+1/2^*\), while
\(r= 2p/(p-4)\) follows from \( 1/r+1/p=1/q\) altogether to \(1/q+1/p=1/2\).
\end{remark}

Let us state our first main theorem, where the density function is only defined a.e. in \(\Omega\).
\begin{theorem}\label{main}
Let the assumptions (H1)-(H5) be fulfilled. For any \(M\in\mathbb{N}\), there exists a triplet 
 \((\rho,\mathbf{u}, \theta)\) such that
\begin{itemize}
\item \(\rho\) is a measurable function satisfying \(\rho\mathbf{u}\in \mathbf{L}^q(\Omega)\),
  with   \(n<q<n+\varepsilon\), for some \(\varepsilon\) depending on \(\Omega\);  
\item \(\mathbf{u}\in \widetilde{\mathbf{u}}_D +\mathbf{V}\);
\item \(\theta\in (\theta_\mathrm{in}+ H^1_\mathrm{in}(\Omega) ) \cap L^\infty (\Omega) \),
\end{itemize}
which is  a weak solution to the NSF problem, with \eqref{boyle} replaced by
\begin{equation}\label{paux}
p_M=T_M( \rho) R_\mathrm{specific}\theta.
\end{equation}
Here, \(T_M\) stands for the truncation, \textit{i.e.} \(T_M(z)=z\) for \(0\leq z\leq M\) and \(T_M\equiv M\) otherwise.
\end{theorem}

Let us state our second main theorem, where the density function is assumed to have \(L^r\)-regularity, for some \(r>2n/(4-n)\) (\(n=2,3\)).
\begin{theorem}\label{main2}
Under the conditions of Theorem \ref{main}, 
 the NSF problem admits at least one solution in \( {L}^r(\Omega) \times (\widetilde{\mathbf{u}}_D + \mathbf{V})\times H^1(\Omega)\)
if provided by \(\rho\in L^r (\Omega) \) satisfying
\begin{equation}\label{arho}
\|\rho\|_{r,\Omega}\leq \mathcal{R},
\end{equation}
for some positive constant \(\mathcal{R}\) independent on \(M\) and \(r\) verifying \eqref{defr}. 
Moreover, the following quantitative estimates
\begin{align}\label{u1}
\|\mathbf{u}-\widetilde{\mathbf{u}}_D\|_{\mathbf{V}} \leq&
\max\left\lbrace  \frac{n}{(n-1)\mu_\#}, \frac{1}{\gamma_\#}\right\rbrace 
\left(  R_4 |\Omega|^{1/2-1/r} + R_1 \|\widetilde{\mathbf{u}}_D\|_{p,\Omega}  \right.\nonumber \\ & \left.
+ \mu^\#\|D\widetilde{\mathbf{u}}_D\|_{2,\Omega}  +\lambda^\#\|\nabla\cdot\widetilde{\mathbf{u}}_D\|_{2,\Omega}\right) \nonumber \\
&   +\sqrt{\frac{\gamma^\#} {\min\left\lbrace\frac{n-1}{n}\mu_\#,\gamma_\#\right\rbrace} }\|\widetilde{\mathbf{u}}_D\|_{2,\Gamma}; \\
\|\theta\|_{1,2,\Omega} \leq& R_2\label{t1}
   \end{align}
hold, where \(R_1\), \(R_2\) and \(R_4\) are defined in \eqref{cotamm}, \eqref{cotaeg} and \eqref{r4},
respectively.
\end{theorem}
  
\begin{remark}
The quantitative estimate \eqref{u1} may be  simplified if, for instance, in the assumption (H5) we assume the existence of 
 \(\widetilde{\mathbf{u}}_D\in  \mathbf{H}^{1}(\Omega)\) having the trace 
\[
\widetilde{\mathbf{u}}_D=\left\{ \begin{array}{ll}
 \mathbf{u}_{D}& \mbox{ on }\Gamma_D\\
\mathbf{0}&\mbox{ on }\Gamma
\end{array}\right.
\] instead.
\end{remark}
 
\section{Strategy}
\label{strat}

Our strategy is based on that the velocity field is not admissible to use for the fixed point argument, because
 the velocity field is not directly measurable and the linear momentum is easier to be physically determined.  

 For fixed \(q>n\) and \(r> 2\), we define  the closed set
 \begin{equation}
   K_{q,r} 
 :=\lbrace \mathbf{m}\in \mathbf{L}^q(\Omega): \ \eqref{defm}\mbox{ holds}\rbrace \times  H^{1}(\Omega) \times  L^r (\Omega)
\label{defv}
 \end{equation}
in the reflexive Banach space
 \(\mathbf{L}^q(\Omega)\times  H^{1}(\Omega)\times L^r (\Omega)\).

 For fixed \(M\in\mathbb{N}\), we build an operator \(\mathcal{T}\)
\begin{align*}
\mathcal{T}:
(\mathbf{m},\xi,\pi)\in   K_{q,r} 
&\mapsto  \mathbf{w}=\mathbf{w}(\mathbf{m},\xi,\pi)\quad\mbox{(Dirichlet--Navier problem)} \\
&\mapsto \mathbf{u}=\mathbf{w}+ \widetilde{\mathbf{u}}_D\\
&\mapsto \rho =\rho (\mathbf{u}) \quad\mbox{(Inlet/outlet problem)} \\
&\mapsto \theta = \theta (\mathbf{m},\xi )\quad (\mbox{Dirichlet--Robin problem)} \\
&\mapsto (\rho\mathbf{u}, \theta, p_M)
\end{align*}
with \(\mathbf{m}\in\mathbf{L}^q(\Omega)\) satisfying
\begin{equation}\label{defm}
\int_\Omega \mathbf{m}\cdot \nabla v \dif{x}=
\int_{\Gamma_D}\rho_\infty\mathbf{u}_D\cdot\mathbf{n} v \dif{s}
,\quad\forall v\in W^{1,q'}(\Omega).
\end{equation} 

Here, we consider three auxiliary problems.
\begin{description}
\item[\textbf{(Dirichlet--Navier problem)}]
The auxiliary velocity \(\mathbf{w}\in\mathbf{V}\) is the unique solution to the Dirichlet--Navier problem defined by
\begin{align}\label{fluidw}
-\int_\Omega \mathbf{m}\otimes \mathbf{w}:\nabla\mathbf{v} \dif{x}+
\int_\Omega \mu(\xi)D\mathbf{w}:D\mathbf{v}\dif{x}+
\int_\Omega\lambda(\xi)\nabla\cdot\mathbf{w}\nabla\cdot\mathbf{v} \dif{x} \nonumber \\
+\int_{\Gamma}\gamma(\xi)\mathbf{w}_T\cdot \mathbf{v}_T\dif{s} =
\int_\Omega \pi\nabla\cdot\mathbf{v} \dif{x}+
\mathcal{G}(\mathbf{m},\xi , \widetilde{\mathbf{u}}_D, \mathbf{v}) ,\ \forall \mathbf{v}\in \mathbf{V},
 \end{align}
 with 
 \begin{align*}
 \mathcal{G}(\mathbf{m},\xi,  \widetilde{\mathbf{u}}_D, \mathbf{v}):=&
 \int_\Omega  \mathbf{m}\otimes \widetilde{\mathbf{u}}_D:\nabla\mathbf{v} \dif{x}
 - \int_{\Gamma}\gamma(\xi)\widetilde{\mathbf{u}}_D\cdot \mathbf{v}_T\dif{s}\\
 & - \int_\Omega \left(
 \mu(\xi)D\widetilde{\mathbf{u}}_D:D\mathbf{v}+\lambda(\xi)\nabla\cdot\widetilde{\mathbf{u}}_D\nabla\cdot\mathbf{v}
 \right)\dif{x}.
 \end{align*}
 
\item[\textbf{(Inlet/outlet problem)}]
The auxiliary density \(\rho\) is a unique solution to the inlet/outlet problem defined by
\begin{equation}\label{syst2}
\int_\Omega \rho \mathbf{u}\cdot \nabla v \dif{x}=
\int_{\Gamma_D}\rho_\infty\mathbf{u}_D\cdot\mathbf{n} v \dif{s},\quad\forall v\in W^{1,q'}(\Omega).
 \end{equation}
 
 \item[\textbf{(Dirichlet--Robin problem)}]
The auxiliary temperature \(\theta-\theta_\mathrm{in} \in H^{1}_\mathrm{in}(\Omega)\) is the unique weak solution to the Dirichlet--Robin problem defined by
\begin{align}
c_v\int_\Omega \mathbf{m}\cdot\nabla\theta v \dif{x}
+&\int_\Omega k(\xi)\nabla \theta \cdot \nabla v \dif{x}  \nonumber \\
 &+\int_{\Gamma_N} h_c(\xi) \theta v  \dif{s}
= \int_{\Gamma_N} h(\xi) v \dif{s},\quad\forall v\in H^{1}_\mathrm{in}(\Omega).\label{newtonpb}
\end{align}
\end{description}
Finally, the auxiliary   pressure is given by \eqref{paux}.

Let us  establish some properties of the linearized convective and advective terms, which are the key points of this paper.
\begin{lemma}\label{lemb}
Let  \( \Omega \subset\mathbb{R}^n\) be a bounded Lipschitz domain. 
For each \( \mathbf{m}\in \mathbf{L}^{q}(\Omega)\), \(q>n\), which verifies \eqref{defm}, the following functionals 
 are well defined and  continuous:
\begin{description}
\item[(convective)] 
\(
\mathbf{u}\in \mathbf{H}^1(\Omega) \mapsto \langle B\mathbf{u},\mathbf{v}\rangle:=
\int_\Omega \mathbf{m}\otimes \mathbf{u}:\nabla\mathbf{v} \dif{x},\) for all \(\mathbf{v}\in \mathbf{V}\).
Moreover, \(B\) is  skew-symmetric in the sense
\begin{equation}\label{skew}
\langle B\mathbf{u}, \mathbf{v}\rangle =-\int_\Omega (\mathbf{m}\cdot \nabla) \mathbf{u}\cdot\mathbf{v} \dif{x}\quad\forall 
\mathbf{u}\in \mathbf{H}^1(\Omega) \ \forall \mathbf{v}\in \mathbf{V}
\end{equation}
 and, in particular, 
\(  \langle B\mathbf{v},\mathbf{v}\rangle=0\) holds for all \(\mathbf{v}\in\mathbf{V}\).
\item[(advective)] \(
e\in H^1(\Omega)\mapsto\int_\Omega \mathbf{m}\cdot\nabla e v \dif{x},\)
for all \(v \in H^1(\Omega).\) Assuming (H5), the relation
\begin{equation}\label{advt}
\int_\Omega \mathbf{m}\cdot\nabla e v \dif{x} =
\int_{\Gamma_D}\rho_\infty\mathbf{u}_D\cdot\mathbf{n} ev \dif{s} -\int_\Omega \mathbf{m}\cdot\nabla v e \dif{x}
\end{equation}
holds for any \(e ,v\in H^1(\Omega).\)
\end{description}
\end{lemma}
\begin{proof}
The wellposedness of each functional is consequence of the H\"older inequality, with exponents \(q\), \(p\) and \(2\) such that
\[
 \frac{1}{2^*}<\frac{1}{p}=\frac{1}{2}-\frac{1}{q} 
\Leftrightarrow q>n,
\]
and the Rellich--Kondrachov embedding \( H^1(\Omega)\hookrightarrow\hookrightarrow L^{p}(\Omega) \) (cf. Remark \ref{rsob}).

The skew symmetry of \(B\), \eqref{skew}, follows from the relation
\[ 
\langle B\mathbf{u}, \mathbf{v}\rangle +\int_\Omega (\mathbf{m}\cdot \nabla) \mathbf{u}\cdot\mathbf{v} \dif{x}
=\int_\Omega \mathbf{m}\cdot \nabla ( \mathbf{u}\cdot\mathbf{v}) \dif{x}
= 
\int_{\Gamma_D}\rho_\infty\mathbf{u}_D\cdot\mathbf{n} ( \mathbf{u}\cdot\mathbf{v}) \dif{s}
\] 
by using \eqref{defm} with \(  \mathbf{u}\cdot\mathbf{v}\in W^{1,q'}(\Omega)\),  \(1<q'<n/(n-1)\).

In \eqref{advt}, the wellposedness of the boundary integral follows from the H\"older inequality, with exponents
\[
\frac{1}{t} + 2 \frac{n-2}{2(n-1)}= 1 \Leftrightarrow t = \left\{ \begin{array}{ll}
n-1&\mbox{if } n=3,4\\
\mbox{arbitrary} &\mbox{if } n=2
\end{array}\right. 
\]
and considering the embedding \(H^1(\Omega) \hookrightarrow L^{2(n-1)/(n-2)}(\partial\Omega)\) and 
\(\rho_\infty\mathbf{u}_D\cdot\mathbf{n} \in L^q(\Gamma_D)\), where \(q>t\). 
\end{proof}

\section{Wellposedness of the Dirichlet--Navier problem}
\label{sZO}

The following properties are well known in the fluid mechanics theory. However, the quantitative estimate is essential in the
fixed point argument and we will fix it.
\begin{proposition}\label{pcotau}
Let the assumptions (H1), (H3) and (H5) be fulfilled.
  For each  (\(\mathbf{m},\xi,\pi)\in  K_{q,2} \), with \(q> n\), let
 \(\mathbf{w}\in \mathbf{V}\) be a solution to the  problem \eqref{fluidw}.
 Then, the following quantitative estimate
\begin{align}\label{cotau}
\min\left\lbrace\frac{n-1}{n}\mu_\#,\gamma_\#\right\rbrace \|\mathbf{w}\|_{\mathbf{V}}^2\leq \frac{n}{(n-1)\mu_\#}
\left( \|\pi\|_{2,\Omega}+\|\mathbf{m}\|_{q,\Omega} \|\widetilde{\mathbf{u}}_D\|_{p,\Omega} \right.\nonumber \\ \left.
+ \mu^\#\|D\widetilde{\mathbf{u}}_D\|_{2,\Omega} 
  +\lambda^\#\|\nabla\cdot\widetilde{\mathbf{u}}_D\|_{2,\Omega}\right)^2
   +\gamma^\#\|\widetilde{\mathbf{u}}_D\|_{2,\Gamma}^2
   \end{align}
holds, with  \(1\leq p< 2^*\) being such that \(1/q+1/p=1/2\).
\end{proposition}
\begin{proof}
This proof is standard, but we sketch it because its quantitative expression.
Choose \(\mathbf{v}=\mathbf{w}\in \mathbf{V}\) as a test function in \eqref{fluidw}, and use Lemma \ref{lemb} to find
\begin{align*}
\int_\Omega \mu(\xi)|D\mathbf{w}|^2\dif{x}+
\int_\Omega\lambda(\xi)|\nabla\cdot\mathbf{w}|^2 \dif{x} 
+\int_{\Gamma}\gamma(\xi)|\mathbf{w}_T|^2\dif{s} \nonumber \\
\leq \left(\|\pi\|_{2,\Omega}+\|\mathbf{m}\|_{q,\Omega} 
\|\widetilde{\mathbf{u}}_D\|_{p,\Omega}+ 2\| \mu(\xi)D\widetilde{\mathbf{u}}_D\|_{2,\Omega}  +\|\lambda(\xi)\nabla\cdot\widetilde{\mathbf{u}}_D\|_{2,\Omega}
\right)\|\nabla\mathbf{w}\|_{2,\Omega} \\
 +\frac{1}{2}\|\sqrt{\gamma(\xi)}\widetilde{\mathbf{u}}_D\|_{2,\Gamma}^2 
 +\frac{1}{2}\|\sqrt{\gamma(\xi)} \mathbf{w}_T\|_{2,\Gamma}^2
\end{align*}
taking the H\"older and Young inequalities into account and
using the fact that \( (\nabla\cdot\mathbf{w})^2\leq|\nabla\mathbf{w}|^2\).
Since \( n\lambda(\xi)+\mu(\xi)\geq 0\),  applying \eqref{mu1} and \eqref{g1} we have
\begin{align*}
\frac{n-1}{n}\mu_\#\|D\mathbf{w}\|_{2,\Omega}^2+\frac{\gamma_\#}{2}\| \mathbf{w}_T\|_{2,\Gamma}^2 \\
\leq\int_\Omega \mu(\xi)\left(|D\mathbf{w}|^2-\frac{1}{n}|\nabla\cdot\mathbf{w}|^2\right) \dif{x} 
+\frac{1}{2}\int_{\Gamma}\gamma(\xi)|\mathbf{w}_T|^2\dif{s} \nonumber \\
\leq  \frac{n-1}{2n}\mu_\# \|D\mathbf{w}\|_{2,\Omega}^2+
\frac{n}{2(n-1)\mu_\#}\left( \|\pi\|_{2,\Omega}+\|\mathbf{m}\|_{q,\Omega} 
\|\widetilde{\mathbf{u}}_D\|_{p,\Omega} \right.\\ \left.+
\mu^\#\|D\widetilde{\mathbf{u}}_D\|_{2,\Omega} 
  +\lambda^\#\|\nabla\cdot\widetilde{\mathbf{u}}_D\|_{2,\Omega}\right)^2
   +\frac{\gamma^\#}{2}\|\widetilde{\mathbf{u}}_D\|_{2,\Gamma}^2.
\end{align*}
Then, readjusting the above estimate we conclude Proposition \ref{pcotau}. 
\end{proof}

The following proposition asserts the existence and uniqueness of auxiliary velocity field.
\begin{proposition}\label{pau}
Let the assumptions (H1), (H3) and (H5) be fulfilled.
  For each  (\(\mathbf{m},\xi,\pi)\in  K_{q,2} \), with \(q> n\),  
  the  problem \eqref{fluidw} admits a unique solution \(\mathbf{w}\in\mathbf{V}\).
\end{proposition}
\begin{proof}
Let  \(a\)  be the (non symmetric) bilinear form on \(\mathbf{V}\times \mathbf{V}\),
 which is associated to  the energy functional \( J:\mathbf{V}\rightarrow\mathbb R\),  defined by
\begin{align*}
J(\mathbf{v})=& 
\int_\Omega \left(\mu(\xi)\frac{|D\mathbf{v}|^2}{2}+\lambda(\xi)\frac{|\nabla\cdot\mathbf{v}|^2}{2}
-\pi\nabla\cdot\mathbf{v} \right)\dif{x} +\\
&+\int_{\Gamma}\gamma(\xi)\frac{|\mathbf{v}_T|^2}{2}\dif{s} -
 \mathcal{G}(\mathbf{m},\xi, \widetilde{\mathbf{u}}_D,\mathbf{v}).
\end{align*}
The existence of \(J'\) is well-defined \cite[Appendix C]{stru} as the Fr\'echet derivative of a Nemytskii operator 
\( F:\Omega\times \mathbb R^n\times\mathbb M^{n\times n}_{\rm sym}\rightarrow\mathbb R\),
and the form \(a\) is sum of the non symmetric and symmetric parts
\[ 
 a(\mathbf{w},\mathbf{v})=-\int_\Omega \mathbf{m}\otimes \mathbf{w}:\nabla\mathbf{v} \dif{x}+\langle J'(\mathbf{w}),\mathbf{v}\rangle.
\] 

 Then, the existence and uniqueness of solution are  consequence of the Lax--Milgram Lemma \cite{lions}. 
\end{proof}

We finalize this section by proving the continuous dependence.
\begin{proposition}[Continuous dependence]\label{pum}
 Let \(\lbrace (\mathbf{m}_m,\xi_m,\pi_m)\rbrace_{m\in\mathbb{N}}\) be a sequence weakly convergent in
    \(  K_{q,r} \), for some \(q> n\) and \(r>2\).  Then, 
 the corresponding solutions \(\mathbf{w}_m=\mathbf{w}(\mathbf{m}_m,\xi_m,\pi_m)\) 
 to the problem \eqref{fluidw}\(_m\),  for each \(m\in\mathbb{N}\), weakly converge to 
 \(\mathbf{w}=\mathbf{w}(\mathbf{m},\xi,\pi)\) in \(\mathbf{V}\), which is the solution
 to the problem \eqref{fluidw} corresponding to the weak limit (\(\mathbf{m},\xi,\pi\)).
 \end{proposition}
 \begin{proof}
Let us  take the sequences
 \begin{align*}
\mathbf{m}_m\rightharpoonup\mathbf{m}&\mbox{ in } \mathbf{L}^q(\Omega) \mbox{ for some } q> n;\\
\xi_m\rightharpoonup\xi&\mbox{ in } H^{1}(\Omega);\\
\pi_m\rightharpoonup\pi&\mbox{ in }L^r(\Omega)\mbox{ for some }r> 2.
\end{align*}
The Rellich--Kondrachov  embeddings \(  H^{1}(\Omega)\hookrightarrow\hookrightarrow L^2(\Omega)\)
and \(  H^{1}(\Omega)\hookrightarrow\hookrightarrow L^2(\Gamma)\)
yield \( \xi_m\rightarrow\xi\) in \( L^2(\Omega)\) and \(L^2(\Gamma)\).
The continuity of the Nemytskii operators, \(\mu\), \(\lambda\) and \(\gamma\), and 
Lebesgue dominated convergence theorem imply
\begin{align*}
\mu(\xi_m)D\mathbf{v}\rightharpoonup\mu(\xi)D\mathbf{v} &\mbox{ in } [L^2(\Omega)]^{n\times n}; \\
\lambda(\xi_m)\nabla\cdot \mathbf{v}\rightharpoonup\lambda(\xi)\nabla\cdot \mathbf{v} &\mbox{ in } L^2(\Omega); \\
\gamma(\xi_m)\mathbf{v}_T\rightharpoonup\gamma(\xi)\mathbf{v}_T &\mbox{ in } \mathbf{L}^2(\Gamma).
\end{align*}

Let \(\mathbf{w}_m=\mathbf{w}(\mathbf{m}_m,\xi_m,\pi_m)\) be the corresponding solution
 to the problem \eqref{fluidw}\(_m\),  for each \(m\in\mathbb{N}\).
  The uniform estimate \eqref{cotau} allows to extract at least one subsequence, still denoted by \(\mathbf{w}_m\),
  of the solutions \( \mathbf{w}_m=\mathbf{w}(\mathbf{m}_m,\xi_m,\pi_m)\) weakly convergent for some
 \( \mathbf{w}\in \mathbf{V}\). Consequently, we have
\begin{align*}
\nabla\mathbf{w}_m\rightharpoonup \nabla\mathbf{w}&\mbox{ in } [L^2(\Omega)]^{n\times n};\\
\mathbf{w}_m\rightarrow \mathbf{w} &\mbox{ in }\mathbf{L}^p(\Omega)\mbox{ and on }\mathbf{L}^2(\partial\Omega),
\end{align*}
 for \(p<2^*\).

The above convergences allow to pass to the limit as  \( m\) tends to infinity in \eqref{fluidw}\(_m\), 
concluding that \( \mathbf{w}\)  satisfies the system \eqref{fluidw}. 
 \end{proof}
 
  \section{Existence and uniqueness of density solution}
\label{sdens}  

In this section, our objective is not to apply the artificial viscosity technique that 
approximates the continuity equation by an elliptic equation through a vanishing viscosity
(also known as elliptic approximation) as it has being usual. Our argument goes out in the spirit of the Helmholtz decomposition 
\[ 
\mathbf{a}=\mathbf{a}_{\bm{\omega}}+\mathbf{a}_\psi ,
\]
where 
\begin{itemize}
\item \(\mathbf{a}_{\bm{\omega}}=\nabla\times {\bm{\omega}}\) stands for the solenoidal (divergence-free) component,
 \textit{i.e.}
 it satisfies \(\nabla\cdot \mathbf{a}_{\bm{\omega}} = \mathbf{0}\) in \(\Omega\).

\item \(\mathbf{a}_\psi =\nabla\psi\) stands for the irrotational component (curl-free),
 \textit{i.e.} it satisfies \(\nabla\times \mathbf{a}_\psi =\mathbf{0}\) in \(\Omega\).
\end{itemize}
We refer to \cite{galdi} for the weak \(L^q\)-solution to the Dirichlet--Laplace problem being motivated by the Weyl decomposition.
 
On the one hand, we consider the  Neumann--Laplace problem 
\begin{align}
\Delta\psi &= 0\qquad \mbox{ in }\Omega\label{NL1}\\
\nabla\psi\cdot\mathbf{n} &=  \rho_\infty \mathbf{u}_{D}\cdot\mathbf{n} \mbox{ on }\Gamma_{D}\\
\nabla\psi\cdot\mathbf{n} &= 0\qquad \mbox{ on }\Gamma,
\label{NL3}
\end{align}
with the zero mean value datum \(g:= \rho_\infty \mathbf{u}_{D}\cdot\mathbf{n}\chi_{\Gamma_{D}}
 \in L^q(\partial\Omega)\hookrightarrow
\left( B^{q'}_{1/q}(\partial\Omega)\right)'\), where the Besov space 
  under the usual notation \( B^{1/q}_{q',q'}\)  
is in fact the Slobodetskii space   \( W^{1/q,q'}\) for \(0<1/q<1/n\) and \(1<q'<n/(n-1)\).  
Thanks to potential theory \cite{fabes,geng-shen}, the problem \eqref{NL1}-\eqref{NL3}, 
with the zero mean value datum \(g\in
B^q_{-1/q}(\partial\Omega)=\left( B^{q'}_{1/q}(\partial\Omega)\right)'\),
 admits the unique solution  \(\psi\in W^{1,q}(\Omega)\) represented by
\[
\psi(x) =\int_{\partial\Omega}G_N(x,y)g(y)\dif{s}_y + \overline{\psi}
\]
where \(G_N(x,y)=E(x-y)+\phi(y)\) is the Green function of the second type, \textit{i.e.} it solves the Neumann--Poisson boundary value problem
\(\Delta G_N(x,\cdot)=\delta_x +1/|\Omega|\) in \(\Omega\) and \(\nabla G_N\cdot\mathbf{n}=0\) on \(\partial\Omega\) \cite{ijpde14}.
Here, \(\delta_x\) is the Dirac delta function at the point \(x\).
The Green function \(E\), being the fundamental solution for \(\Delta\) in \(\mathbb{R}^n\)
with pole at the origin, is given by
\begin{align*}
E(x)&= \frac{1}{2\pi}\ln|x|\ \mbox{ if } n=2;\\
E(x)&= \frac{1}{4\pi}\frac{1}{|x|}\quad\mbox{ if } n=3.
\end{align*}
The function \(\phi\) solves \(\Delta\phi =1/|\Omega|\) 
 in \(\Omega\) and \(\nabla (E(x-\cdot)+\phi)\cdot\mathbf{n}=0\) on \(\partial\Omega\).
The uniqueness of the Neumann problem is possible by the compatibility condition \eqref{cc}, up to the additive constant 
\(\overline{\psi}= |\Omega|^{-1}\int_\Omega \psi \dif{x}\), where \( |\Omega| \triangleq \mathrm{meas} (\Omega)\).

The solution \(\psi\), the so called scalar potential, satisfies the estimate
\begin{equation}\label{cotapsi}
\|\nabla\psi\|_{q,\Omega}\leq C_q \| \rho_\infty \mathbf{u}_{D}\cdot\mathbf{n} \|_{q,\partial\Omega},
\end{equation}
for any \(1< q <\infty\) if \(\Omega\) is of class \(C^1\), for   the sharp ranges   
\(4/3-\varepsilon < q < 4+\varepsilon\) if \(n=2\) or \(3/2-\varepsilon< q <3+\varepsilon\) if \(n=3\) and 
 \(\Omega\) is bounded Lipschitz, with \(\varepsilon>0\)  depending on \(\Omega\) and
 \(C_q>0\) depending on \(n\), \(q\), and the Lipschitz character of \(\Omega\)
  \cite{geng-shen,geng-shen10,mitrea}.
Some specific results are known for convex domains for \(1<q<\infty\) if \(n=2\) and
 for \(1<q<4\) if \(n=3\)  \cite{dong,geng-shen10}.

On the other hand, we find the corresponding vector that makes possible the decomposition.

First, let us establish in the two dimensional space the existence of  our auxiliary density function. 
\begin{proposition}[\(n=2\)]\label{p2D}
Let \((\mathbf{m},\xi,\pi)\in \mathbf{L}^q(\Omega)\times  H^{1}(\Omega)\times L^r (\Omega) 
\) and let \(\mathbf{u}\in \mathbf{H}^1(\Omega)\) be  the corresponding solution
 to \eqref{fluidw} obtained in Section \ref{sZO}.
 Then, there exists  a unique  function 
  \(\rho\) verifying
\begin{equation}\label{drho2}
\rho\mathbf{u}=\nabla \psi +\nabla\times\bm{\omega}\mbox{ a.e. in } \Omega,\end{equation} 
with \(\psi\)  being the unique solution to \eqref{NL1}-\eqref{NL3}
and for some \( \bm{\omega}\) in \(\mathbf{W}^{1,q}(\Omega)\).
In particular, it is non-negative.  Moreover, \eqref{syst2} holds. 
\end{proposition}
\begin{proof}
 Let  \(\psi\in W^{1,q}(\Omega)\)  be the unique solution to \eqref{NL1}-\eqref{NL3}, which verifies \eqref{cotapsi},
 for \(1< q <2+\varepsilon\), with \(\varepsilon>0\)  depending on \(\Omega\) and
 \(C_q>0\) depending on \(n\), \(q\), and the Lipschitz character of \(\Omega\).

 By the potential theory \cite{amrouche,mitrea}, 
 it suffices to seek for  a unique non-negative density function that satisfies 
\begin{equation}\label{weyl}
\rho \mathbf{u} = \nabla\psi +\mathbf{z},
\end{equation}
with \(\mathbf{z}\)  belonging to 
 \[
 \mathbf{H}_q :=\{\mathbf{v}\in \mathbf{L}^{q}(\Omega):\ 
\nabla\cdot\mathbf{v}=0\mbox{ in }\Omega,\ \mathbf{v}\cdot\mathbf{n}=0\mbox{ on }\partial\Omega\}.
 \]

Taking  in \eqref{weyl} the inner product with \(\mathbf{u}\), we  obtain
 \begin{equation}\label{rho2d}
\rho = \frac{1}{|\mathbf{u}|^2} (\nabla\psi+\mathbf{z} )\cdot\mathbf{u}
\quad  \mbox{ in } \Omega[|\mathbf{u}|\not=0],
 \end{equation}
otherwise, we define \(\rho= \rho_0\) in \( \Omega[|\mathbf{u}|=0]\) (cf. Remark \ref{rho0}).
Hereafter, the set \(A[\mathcal{S}]\) means \(\{x\in A:\, \mathcal{S}(x)\}\), with
\(\mathcal{S}\) denoting a sentence to be pointwisely (a.e.) satisfied in \(A\),
  which may represent either \(\Omega\) or \(\Gamma\). 

Taking  in \eqref{weyl} the inner product with  \(\mathbf{u}_\bot=(-u_2,u_1)\in \mathbf{L}^{p}(\Omega)\), 
for any \(1<p<\infty\), we  find the relation
\[
\mathbf{z}\cdot \mathbf{u}_\bot= u_1\partial_2\psi - u_2\partial_1\psi:= \nabla \times \bm{\psi}\cdot\mathbf{u},
\]
taking \(\bm{\psi}=(0,0,\psi)\) into account.
We emphasize that the absolute value of the real number
\(\nabla \times \bm{\psi}\cdot\mathbf{u}/|\mathbf{u}|=\nabla\psi\cdot \mathbf{u}_\bot/|\mathbf{u}|\in \mathbf{L}^q(\Omega)\)
 is the magnitude of the vector rejection of \(\nabla\psi\) from \(\mathbf{u}\),
which is defined as 
\begin{equation}\label{rej}
\nabla\psi -\frac{\nabla\psi\cdot \mathbf{u}_\bot}{|\mathbf{u}|^2} \mathbf{u}_\bot
=\left(\nabla\psi\cdot \frac{\mathbf{u}}{|\mathbf{u}|}\right) \frac{\mathbf{u}}{|\mathbf{u}|},
\end{equation}
where the right hand side stands for  the vector projection of \(\nabla\psi\) onto \(\mathbf{u}\).
 In the following, for the sake of simplicity, we assume that \(\nabla\psi \cdot\mathbf{u}_\bot \geq 0\),
 otherwise we may similarly argue by redefining  \(\mathbf{u}_\bot=(u_2,-u_1)\).
 
Let us consider the following  two cases.

\textsc{Case 1.} If \(\nabla \times \bm{\psi}\cdot\mathbf{u}=0\), it means that 
\[
\nabla\psi = \pm |\nabla\psi|\frac{\mathbf{u}}{|\mathbf{u}|}.
\] 

If \(\angle(\nabla\psi, \mathbf{u})=0\), we may take  \(\mathbf{z}= \mathbf{0} \) in \eqref{weyl}.
 Then, we obtain \(\rho =|\nabla\psi|/ |\mathbf{u}|\).
In particular, \(\rho \) is unique and non-negative.

Notice that the case \(\angle(\nabla\psi, \mathbf{u})=\pi\)  does not occur a.e. in \(\Omega\),
 because it leads to the contradiction 
\[
- |(\nabla\psi)^*| /|\mathbf{u_D}| \mathbf{u_D}\cdot \mathbf{n} =\rho_\infty \mathbf{u_D}\cdot \mathbf{n}
  \mbox{ on }\Gamma_D,
\]
where \((\nabla\psi)^*\) denotes the nontangential maximal function of \(\nabla\psi\) \cite{geng-shen,geng-shen10}.
If  \(\angle(\nabla\psi, \mathbf{u})=\pi\) in an open ball \(B\subset\subset \Omega\), 
we may take  \(\mathbf{z}= -2\nabla\psi \) that fulfills \eqref{weyl} in \(B\).
 Then, we obtain \(\rho =|\nabla\psi|/ |\mathbf{u}|\), which is unique and non-negative.

\textsc{Case 2.}
 If \(\nabla \times \bm{\psi}\cdot\mathbf{u}\not=0\), it means that \(\cos(\angle(\nabla\psi, \mathbf{u}))<1\).  
 Next, we will need three auxiliary functions, denoted by  \(\mathbf{a}\), \(\varphi\)
  and  \(\mathbf{F}\).
 
  In accordance with the latter case,
we define the  vector \(\mathbf{a}\in \mathbf{L}^q(\Omega)\) as follows 
\begin{equation}\label{refle}
\nabla\psi +\mathbf{a}=|\nabla\psi|\frac{\mathbf{u}}{|\mathbf{u}|} :=\rho_1\mathbf{u}.
\end{equation}
This definition captures both cases (a) \(\nabla\psi \cdot\mathbf{u}>0\) and (b) \(\nabla\psi \cdot\mathbf{u}\leq 0\)
(see Fig. \ref{plane2D}).

\begin{figure}
\includegraphics[width=12cm]{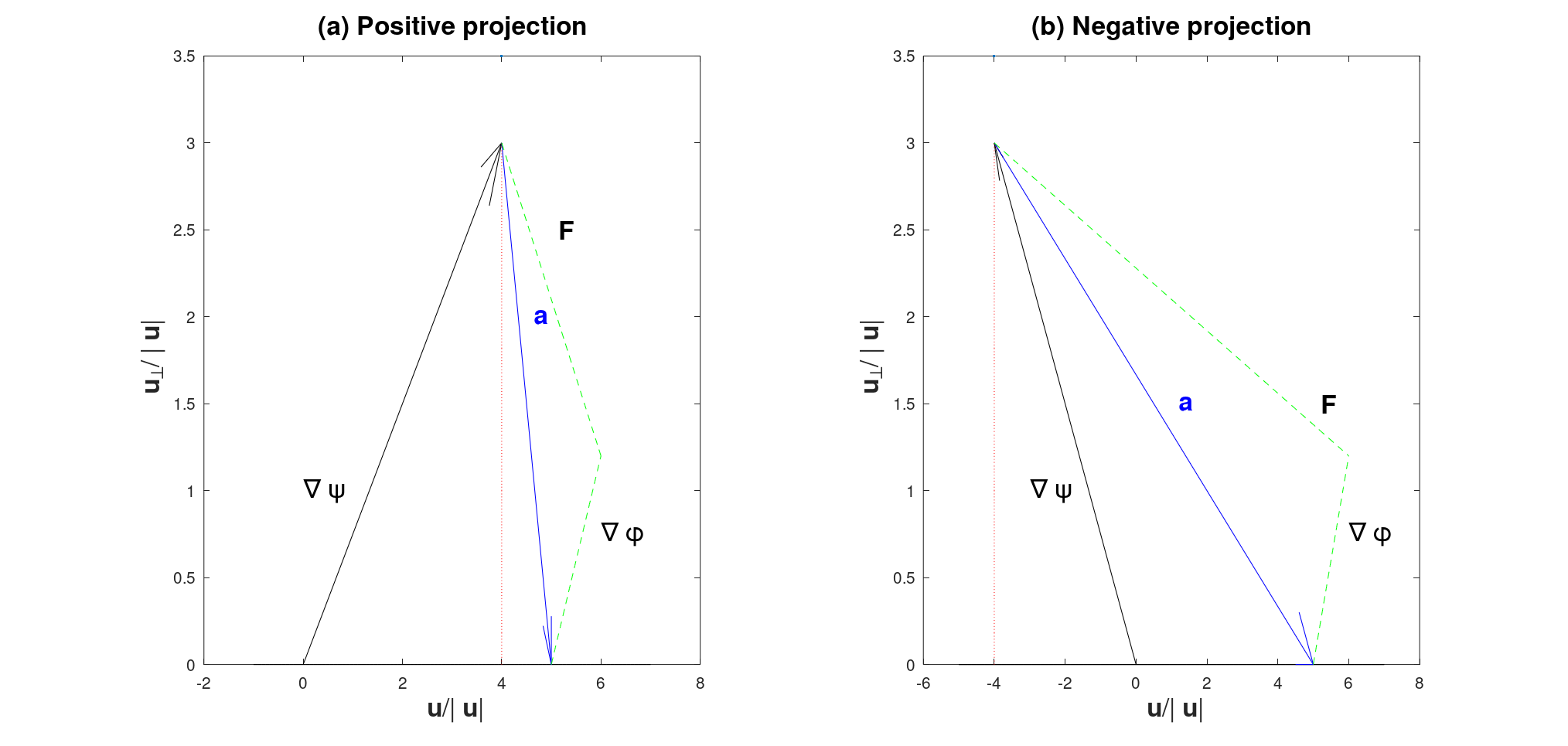}
\caption{Graphical representation of  \(\nabla\psi\) and \(\mathbf{a}\),
in black and blue solid lines, respectively,
relative to the coordinate system (\(\mathbf{u}/|\mathbf{u}|,\mathbf{u}_\bot/|\mathbf{u}|\)).
 Cases (a) \(\nabla\psi \cdot\mathbf{u}>0\) and (b) \(\nabla\psi \cdot\mathbf{u}\leq 0\).
 Red small dashed line represents the rejection, while green
long dashed lines stand for vectors \(\nabla\varphi\) and \(\mathbf{F}\).}
\label{plane2D}
\end{figure}
By the \(L^q-\)Helmholtz--Weyl decomposition (see e.g. \cite{geng-shen10,mitrea}),
the vector \(\mathbf{a}\) may be decomposed as \(\mathbf{a}=\nabla\varphi+ \mathbf{F}\). Here,
 \(\varphi\in W^{1,q}(\Omega)\) is the unique  (up to additive constants) solution to the variational problem
 \begin{equation}\label{NPoisson}
\int_\Omega\nabla\varphi\cdot\nabla v\dif{x} = \int _\Omega\mathbf{a}\cdot\nabla v \dif{x},
\quad \forall v\in W^{1,q'}(\Omega).
\end{equation}
The existence and uniqueness of this scalar potential in the quotient space  \(W^{1,q}(\Omega)/\mathbb{R}\) 
 is guaranteed by the range values of \(q\) for Lipchitz domains.
 Setting
\[
\mathbf{F}=  \mathbf{a}- \nabla\varphi,
\]
it is unique and it belongs to \(\mathbf{H}_q\).  
 Moreover, we have
\begin{equation}\label{cotas}
\max\lbrace \|\nabla\varphi\|_{q,\Omega},\|\mathbf{F}\|_{q,\Omega}\rbrace \leq C_q\|\mathbf{a}\|_{q,\Omega},
\end{equation}
where \(C_q\) depends only on \(q\) and the Lipschitz character of \(\Omega\).

Taking \eqref{refle} and the decomposition \eqref{rej} for \(\varphi\), we have
\begin{align} \label{relation}
\rho_1 \mathbf{u} -\left(
\nabla\varphi\cdot \frac{\mathbf{u}}{|\mathbf{u}|}\right) \frac{ \mathbf{u}}{|\mathbf{u}|}
&=\nabla\psi +\mathbf{F} +\left(
\nabla\varphi\cdot \frac{\mathbf{u}_\bot}{|\mathbf{u}|}\right) \frac{ \mathbf{u}_\bot}{|\mathbf{u}|}\\
\rho_2&= 
   \left\{\begin{array}{ll}
\rho_1-\nabla\varphi\cdot\mathbf{u}/|\mathbf{u}|^2& 
\mbox{ if } \rho_1-\nabla\varphi \cdot\mathbf{u}/|\mathbf{u}|^2>0\\
0  &\mbox{ otherwise}
\end{array}\right.\nonumber
\end{align}
It remains to evaluate the last term of the above relation, namely
\[
\mathbf{f} = \left(
\nabla\varphi\cdot \frac{\mathbf{u}_\bot}{|\mathbf{u}|}\right) \frac{ \mathbf{u}_\bot}{|\mathbf{u}|}.
\]
Or, equivalently
 \begin{align*}
 f_t(t,n)&=0 \\
 f_n(t,n) &=\nabla\varphi\cdot\mathbf{e}_n,
  \end{align*}
by taking  the change of coordinates
 \[
\left[ \begin{matrix}
 \mathbf{e}_t\\
 \mathbf{e}_n
 \end{matrix} \right]=
\left[ \begin{matrix}
 u_1/|\mathbf{u}| & u_2/|\mathbf{u}|\\
- u_2/|\mathbf{u}| & u_1/|\mathbf{u}|
 \end{matrix} \right]
\left[ \begin{matrix}
 \mathbf{e}_1\\
 \mathbf{e}_2
 \end{matrix} \right]
 \]
into account.
Hence, we define \(z_t\) and \(z_n\) such that 
\[
z_n = f_n \quad\mbox{ and }\quad z_t = -\int\partial_n z_n \dif{t} +\rho_3(n),
\] 
with \(\rho_3\) being such that \(z_t<0\). By returning to the Cartesian coordinate system,
 the boundary condition \(\mathbf{z}\cdot\mathbf{n}=0\)
guarantees the uniqueness of \(\rho_3\).
Finally, we choose the  vector \(\mathbf{z}\in\mathbf{H}_q\) such that 
\[
\mathbf{z} = \mathbf{F} + z_t\mathbf{e}_t+z_n\mathbf{e}_n ,
\]
by recalling the vector \(\mathbf{F}\) from  \eqref{relation}.
 
 From  \eqref{drho2},
the function \(\rho\) verifies the variational formulation \eqref{syst2},
  which concludes the proof of Proposition \ref{p2D}. 
\end{proof}

\begin{remark}\label{rho0}
 We call by \(\rho_0\) the constant density at STP (standard temperature and pressure).
Notice that the velocity may be zero, the so called stagnation.
The no upper boundedness of the density is related to that
 \(|\mathbf{u}|\rightarrow 0\) means \( \rho\rightarrow\infty\). However,
neither the velocity function nor the density function are continuous.
 It suggests that some upper boundedness will be possible, but it is still an open problem.
\end{remark}

Next, we study the three dimensional space.
\begin{proposition}[\(n=3\)]\label{p3D}
Let \((\mathbf{m},\xi,\pi)\in   \mathbf{L}^q(\Omega)\times  H^{1}(\Omega)\times L^r (\Omega)
\) and let \(\mathbf{u}\) be  the corresponding solution
 to \eqref{fluidw} obtained in Section \ref{sZO}.
 Then, there exists a unique  function \(\rho\) verifying
\begin{equation}\label{drho3}
\rho\mathbf{u}=\nabla \psi +\nabla\times\bm{\omega}   \mbox{ a.e. in } \Omega,
\end{equation}  
with \(\psi\) being the unique solution to \eqref{NL1}-\eqref{NL3} and for some  \( \bm{\omega}\)
 in \(\mathbf{W}^{1,q}(\Omega)\).
In particular, it is non-negative. Moreover, \eqref{syst2} holds.
\end{proposition}
\begin{proof}
 Let  \(\psi\in W^{1,q}(\Omega)\)  be the unique solution to \eqref{NL1}-\eqref{NL3}, which verifies \eqref{cotapsi},
 for \(3/2-\varepsilon< q <3+\varepsilon\), with \(\varepsilon>0\)  depending on \(\Omega\) and
 \(C_q>0\) depending on \(n\), \(q\), and the Lipschitz character of \(\Omega\).
  
In the three dimensional space, arguing as in the two dimensional Proposition \ref{p2D} we seek for
\begin{equation}\label{defrho}
\rho=(\nabla \psi +\nabla\times\bm{\omega})\cdot \mathbf{u}/|\mathbf{u}|^2\quad  \mbox{ in } \Omega[|\mathbf{u}|\not=0]
\end{equation}
otherwise, we define \(\rho= \rho_0\) if \( \Omega[|\mathbf{u}|=0]\)  (cf. Remark \ref{rho0}).  
As the objective is to find a scalar function,  the argument of the proof of Proposition \ref{p2D} may be repeated in the
plane formed by the vectors \(\mathbf{u}\) and \(\nabla\psi\), \textit{i.e.}
we consider the local  coordinate system
  \((\mathbf{e}_t,\mathbf{e}_n,\mathbf{0})\), where
 \(\mathbf{e}_t = \mathbf{u}/|\mathbf{u}|\) and  \(\mathbf{e}_n = \nabla\psi\times\mathbf{u}/|\nabla\psi\times\mathbf{u}|\).

Therefore,  there exists a vector potential \(\bm{\omega}\)
such that \(\nabla\times \bm{\omega} =\rho \mathbf{u}-\nabla\psi\), \textit{i.e.}   \eqref{drho3},
which may be given unique \cite{amrouche}.
\end{proof}
    
Finally, we are in conditions to determine the estimate for the linear momentum (cf. \eqref{cotapsi}). 
\begin{corollary}\label{crho}
 Let  \(\Omega\) be Lipschitz. 
For \(n = 2,3\), let \(\rho\) be the unique    function given at Propositions \ref{p2D} and \ref{p3D}.
Then,  the estimate 
\begin{equation}\label{cotamm}
\|  \rho\mathbf{u} \|_{q,\Omega}\leq 3 C_q \| \rho_\infty \mathbf{u}_{D}\cdot\mathbf{n} \|_{q,\partial\Omega}:=R_1
\end{equation}
holds, 
for any \(n< q <n+\varepsilon\) and \(n=2,3\),
 with  \(\varepsilon\) depending on \(\Omega\)
 and \(C_q\) depending on \(n\), \(q\), and the Lipschitz character of \(\Omega\).
\end{corollary}

 \section{Wellposedness of the Dirichlet--Robin   problem}
 \label{sSOLA}
 
 The existence of the solution \(\theta\in H^{1}(\Omega)\), which satisfies \eqref{tin},
  to the problem \eqref{newtonpb} is stated in the  following proposition.
 \begin{proposition}[Existence and uniqueness]\label{pae}
Let the assumptions (H2) and (H4)-(H5) be fulfilled.
  For each \((\mathbf{m},\xi)\in \mathbf{L}^q(\Omega)\times  H^{1}(\Omega)\),
 which  verifies \eqref{defm},
  the problem \eqref{newtonpb} admits a unique solution \(\theta\in H^1(\Omega)\) such that
 \(\theta=\theta_\mathrm{in}\) on \(\Gamma_\mathrm{in}\). Moreover,  the estimate
\begin{equation}\label{cotaeg}
\|\nabla \theta\|_{2,\Omega}^2+ \|\theta\|_{2,\Gamma}^2
\leq \frac{h^\#}{\min\lbrace 2k_\#, h_\#\rbrace } \|\theta_\mathrm{in} + \theta_\mathrm{e}\|_{2,\Gamma_N}^2 := R_2^2
\end{equation}
holds.
 \end{proposition}
 \begin{proof}
 The existence and uniqueness of \( \theta=u+\theta_\mathrm{in}\), with \(u\in H^1_\mathrm{in}(\Omega)\), solving \eqref{newtonpb} is standard by the
 Lax--Milgram Lemma. The problem \eqref{newtonpb} reads
 \[
 a(u,v)= \int_{\Gamma_N} h_c(\xi)(\theta_\mathrm{e}-\theta_\mathrm{in}) v \dif{s},
 \quad\forall v\in H^{1}_\mathrm{in}(\Omega),
 \] 
 where the continuous bilinear form \(a\) from \(H^{1}_\mathrm{in}(\Omega)\times H^{1}_\mathrm{in}(\Omega)\)
  into \(\mathbb{R}\),
is defined by
 \[
 a(u,v)=c_v\int_\Omega \mathbf{m}\cdot\nabla u v \dif{x}
+\int_\Omega k(\xi)\nabla u \cdot \nabla v \dif{x}  +\int_{\Gamma_N} h_c(\xi) u v  \dif{s}.
 \]
 Moreover, using the assumptions \eqref{defchi} and \eqref{h1}-\eqref{hout},  the form \(a\) is coercive:
 \begin{align*}
 a( u,u)= c_v\int_\Omega \mathbf{m}\cdot\nabla(u ^2/2) \dif{x}
+\int_\Omega k(\xi) |\nabla u |^2 \dif{x} 
 +\int_{\Gamma_N} h_c(\xi) u^2  \dif{s}  \nonumber \\
 \geq
 \min\lbrace k_\#,h_\#\rbrace \left(\|\nabla u\|_{2,\Omega}^2+\|u\|_{2,\Gamma}^2\right),
 \end{align*}
 taking \(u^2\in W^{1,q'}(\Omega)\)  into account, that is, \eqref{defm} reads
 \[
 \int_\Omega \mathbf{m}\cdot\nabla (u^2/2) \dif{x}=\int_{\Gamma_\mathrm{out}}\rho_\infty\mathbf{u}_D\cdot \mathbf{n}
 u^2/2\dif{s}\geq 0.
 \] 
 The estimate \eqref{cotaeg} follows by choosing  \(v=\theta-\theta_\mathrm{in}\) as a test function in  \eqref{newtonpb},
  arguing as above, considering that \(\nabla u=\nabla\theta\) and 
  \[
k_\#\|\nabla \theta\|_{2,\Omega}^2+\frac{1}{2}\|\sqrt{h_c(\xi)}  \theta\|_{2,\Gamma_N}^2
\leq\frac{1}{2}\left(
\|\sqrt{h_c(\xi)} (\theta_\mathrm{in} +  \theta_\mathrm{e} ) \|_{2,\Gamma_N}^2\right)
\]
after routine computations. 
 \end{proof}
 
The following minimum-maximum principle is standard, its proof argument differs on the advective and boundary terms. For reader convenience,
we provide the proof.
\begin{proposition}[Minimum-maximum principle]\label{maxmin}
Let \(\theta\in H^{1}(\Omega)\) be a solution to the problem \eqref{newtonpb}.
Then, the lower and upper bounds
\begin{equation}\label{tmax}
\mathrm{ess}\inf_{\partial\Omega}\theta_0 \leq \theta \leq \mathrm{ess}\sup_{\partial\Omega}\theta_0\mbox{ a.e. in } \Omega
\end{equation}
hold.
\end{proposition}
\begin{proof}
Let us define \(T_\mathrm{min}=\mathrm{ess}\inf\lbrace \theta_0(x):\, x\in\partial\Omega\rbrace\).
Let us choose \(\phi(\theta)=(\theta-T_\mathrm{min})^-=\min\lbrace \theta-T_\mathrm{min},0\rbrace\in H^1_\mathrm{in}(\Omega) \) 
as a test function in \eqref{newtonpb}.
Applying the assumptions \eqref{defchi} and \eqref{h1}-\eqref{hout}, we have
\[
\int_\Omega \mathbf{m}\cdot\nabla\theta \phi(\theta) \dif{x}+ 
k_\#\|\nabla \theta\|_{2,\Omega[\theta<T_\mathrm{min}]}^2 +
h_\#\| \theta-T_\mathrm{min}\|_{2,\Gamma[\theta<T_\mathrm{min}]}^2 \leq  0.
\] 
 Since  the advective term verifies 
 \begin{align*}
\int_\Omega \mathbf{m}\cdot\nabla\theta \phi(\theta) \dif{x}&=
\int_{\Omega[\theta<T_\mathrm{min}]}\mathbf{m}\cdot \nabla (\phi^2(\theta)/2) \dif{x} \\
&=\int_{\Gamma_D}\rho_\infty\mathbf{u}_D\cdot\mathbf{n} \phi^2(\theta)/2 \dif{s} =   \int_{\Gamma_\mathrm{out}} \rho_\infty u_\mathrm{out} \phi^2(\theta)/2 \dif{s} \geq 0 ,
 \end{align*}
  taking \eqref{defm}  and next \eqref{h1}-\eqref{hout} into account, we deduce
   \begin{equation}  \label{eqeq1}
k_\#\|\nabla \phi(\theta)\|_{2,\Omega}^2 
+h_\#\| \phi(\theta)\|_{2,\Gamma}^2  \leq  0.
\end{equation}
Then, we conclude that \(\phi(\theta)=0\) in \(\Omega\), which means that the lower bound is proved.

The upper bound is analogously proved, by defining \(T_\mathrm{max}=\mathrm{ess}\sup\lbrace \theta_0(x):\, x\in\partial\Omega\rbrace\) and
 choosing \(\phi(\theta)=(\theta-T_\mathrm{max})^+=\max\lbrace \theta-T_\mathrm{max},0\rbrace\in H^1_\mathrm{in}(\Omega) \) 
as a test function in \eqref{newtonpb}. 
\end{proof}

We finalize this section by proving the continuous dependence.
\begin{proposition}[Continuous dependence]\label{pem}
 Let \(\lbrace (\mathbf{m}_m,\xi_m)\rbrace_{m\in\mathbb{N}}\) be a  weakly convergent sequence
  in \(\mathbf{L}^q(\Omega)\times H^{1}(\Omega)\), for some \(q>n\).
  Then, 
 the corresponding solutions \(\theta_m=\theta(\mathbf{m}_m,\xi_m)\in H^{1}(\Omega)\) 
 to the problem \eqref{newtonpb}\(_m\),  for each \(m\in\mathbb{N}\), weakly converge to 
 \(\theta_m=\theta(\mathbf{m},\xi)\), which is the solution
 to the problem \eqref{newtonpb} corresponding to the weak limit (\(\mathbf{m},\xi\)).
 \end{proposition}
 \begin{proof}
Let us  take the sequences
 \begin{align*}
\mathbf{m}_m\rightharpoonup\mathbf{m}&\mbox{ in } \mathbf{L}^q (\Omega);\\
\xi_m\rightharpoonup\xi&\mbox{ in } H^{1}(\Omega).
\end{align*}
The Rellich--Kondrachov embeddings \(  H^{1}(\Omega)\hookrightarrow\hookrightarrow L^2(\Omega)\) 
and \(  H^{1}(\Omega)\hookrightarrow\hookrightarrow L^2(\partial\Omega)\)
yield \( \xi_m\rightarrow\xi\) in \( L^2(\Omega)\) and \(L^2(\partial\Omega)\).

By the one hand, from \( \xi_m\rightarrow\xi\) in \( L^1(\Omega)\) and a.e. in \( \Omega\),
and  the assumption \eqref{defchi}, 
the continuity property of the Nemytskii  operator associated to the leading coefficient \( k\) implies that 
\begin{align*}
k(\cdot,\xi_m)\rightarrow k(\cdot,\xi) &\mbox{ a.e. in }\Omega;\\
k(\xi_m)\nabla v \rightarrow k(\xi)\nabla v &\mbox{ in } \mathbf{L}^2(\Omega).
\end{align*} 

By the other hand, from \( \xi_m\rightarrow\xi\) in \( L^1(\partial\Omega)\) and a.e. on \( \partial\Omega\),
and the assumptions \eqref{defhm}-\eqref{hout}, 
the continuity property of the Nemytskii operator associated to the boundary coefficient \( h_c\) implies that
\begin{align*}
h_c(\cdot,\xi_m)\rightarrow h_c(\cdot,\xi) &\mbox{ a.e. on }\Gamma_N;\\
h_c(\xi_m) v \rightarrow h_c(\xi) v &\mbox{ in }L^2(\Gamma_N).
\end{align*} 

For each \(m\in\mathbb{N}\), let \(\theta_m=\theta(\mathbf{m}_m,\xi_m)\) be the corresponding solution 
to the problem \eqref{newtonpb}\(_m\). 
 The uniform estimate \eqref{cotaeg} allows to  extract at least one subsequence, still denoted by \(\theta_m\),
  of the solutions \(\theta_m=\theta(\mathbf{m}_m,\xi_m)\) weakly convergent for some
 \( \theta\in H^{1}(\Omega)\). 
 
 The above convergences do not be sufficient to
  the passage to the limit, as  \( m\) tends to infinity, in \eqref{newtonpb}\(_m\). It remains to pass
  the advective term to the limit. To this aim, we prove the following strong convergence
  \(  \nabla \theta_m\rightarrow\nabla \theta  \) in \( L^2 (\Omega)\). Arguing as in \cite{m3as2006},
  we apply the assumption \eqref{defchi}  and we decompose to obtain
  \[
  k_\# \int_{\Omega}|\nabla(\theta_m -\theta)|^2\mathrm{d}x\leq 
  \int_{\Omega} (k(\xi_m)\nabla\theta_m -k(\xi_m)\nabla\theta)\cdot\nabla(\theta_m -\theta)\dif{x} =  
  \mathcal{I}_1 - \mathcal{I}_2,
  \]
  with
  \begin{align*}
  \mathcal{I}_1 &= \int_{\Omega} k(\xi_m)\nabla\theta_m\cdot\nabla (\theta_m -\theta)\dif{x}\\
  \mathcal{I}_2 &=  \int_{\Omega} k(\xi_m)\nabla\theta\cdot\nabla (\theta_m -\theta)\dif{x}
  \longrightarrow 0 \mbox{ as } m\rightarrow \infty.
  \end{align*}
  
Next, to prove that \(\mathcal{I}_1\) also tends to zero,  we take
  \(v= \theta_m-\theta \)  as a test function in \eqref{newtonpb}\(_m\). Hence, we obtain
   \begin{align*}
 \int_{\Omega} \mathbf{m}_m \cdot \nabla\frac{(\theta_m -\theta)^2}{2} \dif{x}+ \mathcal{I}_1 = 
  \int_{\Omega} \mathbf{m}_m\cdot \nabla\theta (\theta_m- \theta) \dif{x} & \\
 +  \int_{\partial\Omega} (h(\xi_m) - h_c(\xi_m)\theta_m) (\theta_m -\theta)\dif{s}
   \longrightarrow 0&\mbox{ as } m\rightarrow \infty,
  \end{align*}
  taking the Rellich--Kondrachov embeddings \( H^1(\Omega)\hookrightarrow\hookrightarrow L^p (\Omega)\), with \(p<2^*\),
  and \( H^1(\Omega)\hookrightarrow\hookrightarrow L^2 (\partial\Omega)\) 
  into account for \(n=2,3\). Then,  applying the relation \eqref{defm} into the left hand side of the above equality,  
  we find the claim, \textit{i.e.} the strong convergence.

Then, the passage to the limit yields  that \(\theta\)  satisfies \eqref{newtonpb}, 
 concluding Proposition \ref{pem}. 
 \end{proof}

 \section{Existence of a fixed point to the problem ({\sc Proof of Theorem \ref{main}})}
\label{smain1}

 We will apply the following Tychonoff extension to weak topology of the Schauder fixed point
theorem \cite[pp. 453-456 and 470]{dsch}.
\begin{theorem}\label{fpt}
Let \( K\) be a nonempty weakly sequential compact convex subset of a locally convex linear topological vector space \( V\).
 Let \( \mathcal{T}:K\rightarrow K\) be a weakly sequential continuous  operator.
  Then \( \mathcal{T}\) has at least one fixed point.
\end{theorem}

Let    \(V= \mathbf{L}^q(\Omega)\times  H^{1}(\Omega)\times L^r (\Omega)\) and \(K_{q,r}\) 
be  the nonempty  convex set  defined in \eqref{defv}. 
  We define \( K = K_{q,r}\cap B\), where \(B\) is 
 the closed (bounded) ball, with radius \(R_1,R_2,R_3>0\) defined in \eqref{cotamm}, \eqref{cotaeg} and \eqref{cotapm}, respectively.
 In the reflexive Banach space \(V\),
the closed, convex and bounded set  \(K\) is compact for the weak topology \(\sigma (V,V')\), \textit{i.e.} 
it is weakly sequential compact.

Let \(\mathcal{T}\) be the operator defined in Section \ref{strat}.
 The fixed point argument (cf. Theorem \ref{fpt}) guarantees the existence of the required solution, by proving
 the following two propositions, namely, Propositions \ref{pt1} and  \ref{pt2}.
\begin{proposition}\label{pt1}
Let the assumptions (H1)-(H5) be fulfilled.
Then, the operator \( \mathcal{T}\) is  well defined and it maps \(K\) into itself.
\end{proposition}
 \begin{proof}
 The well-definiteness of \(\mathcal{T}\) is consequence of Proposition \ref{pau}, Corollary \ref{crho}, and
 Proposition \ref{pae}.
  In order to prove that \( \mathcal{T}\) maps  \(K\) into itself, let \( (\mathbf{m},\xi,\pi)\in K\) and
  \[
  \mathcal{T}(\mathbf{m},\xi,\pi) =  (\rho\mathbf{u}, \theta, p_M) .
  \]
  That is, we seek for \(R_1,R_2,R_3>0\) such that 
  \begin{align*}
  \|\mathbf{m}\|_{q,\Omega}\leq R_1,\qquad \|\xi\|_{1,2,\Omega}\leq R_2, \qquad \|\pi\|_{r,\Omega}\leq R_3;\\
  \|\rho\mathbf{u}\|_{q,\Omega}\leq R_1,\qquad \|\theta\|_{1,2,\Omega}\leq R_2, \qquad \|p_M\|_{r,\Omega}\leq R_3.
  \end{align*}
  
  Thanks to Corollary \ref{crho}, the quantitative estimate \eqref{cotamm} guarantees the existence of \(R_1\), for \(q\)
depending on the smoothness of the domain \(\Omega\).
     Thanks to Proposition \ref{pae}, the quantitative estimate \eqref{cotaeg} guarantees  the existence of \(R_2\).

The existence of \(R_3\) is due to the definition \eqref{paux}, we concretely have
\begin{equation}\label{cotapm}
\|p_M\|_{r,\Omega}\leq M |\Omega|^{1/r} R_\mathrm{specific} \mathrm{ess}\sup_{\partial\Omega}\theta_0 :=R_3,
\end{equation}
by considering the estimate \eqref{tmax}. 
 \end{proof}
 
 \begin{proposition}\label{pt2}
Let the assumptions (H1)-(H5) be fulfilled.
Then, the operator \( \mathcal{T}\) is weakly sequential continuous.
\end{proposition}
 \begin{proof} 
 Let  \(\lbrace (\mathbf{m}_m,\xi_m,\pi_m)\rbrace_{m\in\mathbb{N}}\) be a sequence of  \(V\)
 weakly convergent to (\(\mathbf{m},\xi,\pi) \), namely
 \begin{align*}
\mathbf{m}_m\rightharpoonup\mathbf{m}&\mbox{ in }\mathbf{L}^q(\Omega);\\
\xi_m\rightharpoonup\xi&\mbox{ in } H^{1}(\Omega);\\
\pi_m\rightharpoonup\pi&\mbox{ in }L^r(\Omega).
\end{align*}
 
Thanks to Proposition \ref{pum}, the corresponding solutions \(\mathbf{w}_m=\mathbf{w}(\mathbf{m}_m,\xi_m,\pi_m) \in \mathbf{V} \) 
 to the problem \eqref{fluidw}\(_m\),  for each \(m\in\mathbb{N}\), 
 weakly converge to the  solution \(\mathbf{w}=\mathbf{w}(\mathbf{m},\xi,\pi)\)  to the problem \eqref{fluidw}.
  Thus, we get
 \[
\mathbf{u}_m\rightharpoonup \mathbf{u}\mbox{ in }\mathbf{H}^1(\Omega) .
\] 
Consequently, we get \(\mathbf{u}_m\rightarrow \mathbf{u}\) a.e. in \(\Omega\).
Notice that \(\mathbf{u}\) satisfies
\begin{align*}
-\int_{\Omega}  \mathbf{m}\times \mathbf{u}:\nabla \mathbf{v}\dif{x} 
+\int_{\Omega}\mu(\xi)D\mathbf{u}:D\mathbf{v} \dif{x} 
+\int_{\Omega} \lambda(\xi)\nabla\cdot\mathbf{u}\nabla\cdot\mathbf{v} \dif{x}  \\
+\int_{\Gamma}\gamma(\xi)\mathbf{u}_T\cdot \mathbf{v}_T\dif{s} 
 =\int_{\Omega}  \pi\nabla \cdot\mathbf{v}\dif{x},
\end{align*}
for all \(\mathbf{v}\in \mathbf{V}\), and the convective term verifies \eqref{skew}.

Let \(\rho_m\) be the unique solution given at Propositions \ref{p2D} and \ref{p3D},  for \(n = 2,3\), respectively. 
Then, it follows that \(\rho_m\) a.e. converges to \(\rho\) in \(\Omega\).
Thanks to Corollary  \ref{crho}, we have
\(\rho_m\mathbf{u}_m\rightharpoonup\rho\mathbf{u}\)  in \(\mathbf{L}^q(\Omega)\),
which limit satisfies \eqref{syst2}.

Thanks to Proposition \ref{pem},
 the corresponding solutions \(\theta_m=\theta(\mathbf{m}_m,\xi_m)\) 
to the problem \eqref{newtonpb}\(_m\),  for each \(m\in\mathbb{N}\),  weakly converge to the  solution
\(\theta=\theta(\mathbf{m},\xi)\) in \(H^1(\Omega)\).
Thus, \(\theta_m\) strongly converges to \(\theta\) in \(L^p(\Omega)\), for \(1<p<2^*\). Thanks to
\eqref{cotapm} and the Lebesgue dominated convergence theorem, we have
\[
T_M(\rho_m)\theta_m \rightharpoonup T_M(\rho)\theta  \mbox{ in }L^r(\Omega).
\]
Then,  the operator \( \mathcal{T}\) is weakly sequential continuous, which finishes the proof of Proposition \ref{pt1}. 
 \end{proof}
 
 Therefore, we are in condition to obtain the fixed point
  \[
  (\mathbf{m},\xi,\pi) = (\rho\mathbf{u}, \theta, p_M),
  \]
which is the required solution. 
Finally,
the argument of Proposition \ref{maxmin}, with the auxiliary problem \eqref{newtonpb} being replaced by the variational problem \eqref{heatw},
 can be applied to obtain the \(L^\infty\)-regularity of the temperature \(\theta\),
and the proof of Theorem \ref{main} is concluded.

\section{Passage to the limit as \(M\rightarrow\infty\)
({\sc Proof of Theorem \ref{main2}})}
\label{smain2}

The proof of the main result is due to compactness arguments.

Under the assumption \eqref{arho},
the solution \((\rho_M,\mathbf{u}_M,\theta_M)\)  determined in Theorem \ref{main} satisfies
\begin{align}
\|\rho_M\|_{r,\Omega}\leq \mathcal{R}; \\
\|\rho_M\mathbf{u}_M\|_{q,\Omega}\leq R_1; \label{r1}\\
\|\theta_M\|_{1,2}\leq R_2, \label{r3}
\end{align}
considering \(R_1\) and \(R_2\) from \eqref{cotamm} and \eqref{cotaeg}, respectively.
 Arguing as in \eqref{cotapm}
with \(\mathcal{R}\) replacing \(M |\Omega|^{1/r} \), we get 
\begin{equation}\label{r4}
\|p_M\|_{r,\Omega}\leq \mathcal{R} R_\mathrm{specific} \mathrm{ess}\sup_{\partial\Omega}\theta_0 :=R_4.
\end{equation}
Hence, we can extract a subsequence of \(p_M\), still labeled by \(p_M\), weakly convergent to \(p\) in \(L^r(\Omega)\).

Considering \eqref{r1} and \eqref{r4}, the estimate \eqref{cotau}  reads
\begin{align}
\min\left\lbrace\frac{n-1}{n}\mu_\#,\gamma_\#\right\rbrace \|\mathbf{w}\|_{\mathbf{V}}^2\leq \frac{n}{(n-1)\mu_\#}
\left( R_4 |\Omega|^{1/2-1/r} + R_1 \|\widetilde{\mathbf{u}}_D\|_{p,\Omega} \right.\nonumber \\ \left.
+ \mu^\#\|D\widetilde{\mathbf{u}}_D\|_{2,\Omega} 
  +\lambda^\#\|\nabla\cdot\widetilde{\mathbf{u}}_D\|_{2,\Omega}\right)^2
   +\gamma^\#\|\widetilde{\mathbf{u}}_D\|_{2,\Gamma}^2.\label{cotau4}
   \end{align}
Then, the convergences
 \begin{align*}
\rho_M\rightharpoonup \rho &\mbox{ in } L^r (\Omega); \\
\mathbf{u}_M\rightharpoonup\mathbf{u}&\mbox{ in } \mathbf{H}^1(\Omega)  ;\\
\theta_M\rightharpoonup\theta &\mbox{ in } H^{1}(\Omega),
\end{align*}
hold, as \(M\) tends to infinity.
From the above convergences, we identify the limit \[p=\rho R_\mathrm{specific}\theta.
\]

The quantitative estimates \eqref{u1}-\eqref{t1} are established from the estimates \eqref{cotau4} and \eqref{cotaeg}, respectively.
Therefore,  the proof of Theorem \ref{main2} is concluded.

\subsection*{Acknowledgements.}
This preprint is a submitted manuscript. 
The Version of Record of this
article is published in S\~ao Paulo Journal of Mathematical Sciences, and is available online at 
https://doi.org/10.1007/s40863-021-00262-z


\begin{thebibliography}{99}

 
\bibitem{amrouche}
C. Amrouche, N. Seloula,
 \(L^p-\)theory for vector potentials and Sobolev’s inequalities for vector fields. Application to the Stokes equations with
 pressure boundary conditions,
 \textit{Math. Mod. Meth. Appl. Sci.} \textbf{23}  (2013), 37-92.

\bibitem{bveiga87}
H. Beir\~ao da Veiga,
Existence results in Sobolev spaces for a stationary transport equation,
\textit{Ric. Mat.}  \textbf{36} (1987), 173-184.

\bibitem{brezina}
J. B\v{r}ezina, A. Novotn\'y,
On weak solutions of steady Navier-Stokes equations for monatomic gas,
\textit{Comment. Math. Univ. Carolin.} \textbf{49} :4 (2008),  611-632.

\bibitem{chung}
S.R. Chung, C.H. Suh, J.H. Baek, H.S. Park, Y.J. Choi, J.H. Lee,
 Safety of radiofrequency ablation of benign thyroid nodules
and recurrent thyroid cancers: a systematic review and meta-analysis, 
\textit{Int. J. Hyperthermia} \textbf{33} :8 (2017), 920-930.

\bibitem{m3as2006}
L. Consiglieri, 
{Steady-state flows of thermal viscous incompressible fluids with convective-radiation effects}, 
\textit{Math. Mod. Meth. Appl. Sci.} \textbf{16} :12 (2006), 2013-2027.
 

 \bibitem{ijpde14}
L. Consiglieri, Explicit estimates for solutions of mixed elliptic problems,
 \textit{Int. J. Partial Differential Equations} \textbf{2014}  (2014), Article ID 845760. https://doi.org/10.1155/2014/845760

\bibitem{lap2011}
L. Consiglieri, 
{\em Mathematical analysis of selected problems from fluid thermomechanics.
The  \( (p-q)\) coupled fluid-energy systems.}  Lambert Academic Publishing, Saarbr\"ucken 2011.

\bibitem{dong}
H. Dong,
       On elliptic equations in a half space or in convex wedges with irregular coefficients.
\textit{Adv. Math.} \textbf{238} (2013), 24-49.

\bibitem{sarka2009}
B. Ducomet, S. Ne\v{c}asov\'a, A. Vasseur,
On spherically symmetric motions of a viscous compressible barotropic and selfgravitating gas,
\textit{J. Math. Fluid Mech.} \textbf{13}  (2011), 191-211. 

\bibitem{dsch}
 N. Dunford, J.T. Schwartz, \textit{Linear operators, Part I}
 Interscience Publ., New York 1958.

\bibitem{fabes}
E. Fabes, M. Jodeit Jr., N. Rivi\'ere, 
Potential techniques for boundary value problems on \(C^1\) domains,
\textit{Acta Math.} \textbf{141} (1978), 165-186.

\bibitem{frehse-w}
J. Frehse, M. Steinhauer, W. Weigant,
The Dirichlet problem for steady viscous compressible flow in three dimensions,
\textit{J. Math. Pures Appl.} \textbf{97} (2012), 85-97.

\bibitem{galdi}
G.P. Galdi, C.G. Simader, 
Existence, uniqueness and \(L^q\) -estimates for the Stokes problem in an exterior domain,
\textit{Arch. Rational Mech. Anal.} \textbf{112} (1990), 291-318.

\bibitem{geng-shen}
J. Geng, Z. Shen,
 The \(L^p\) boundary value problems on Lipschitz domains, 
\textit{Adv. Math.} \textbf{216}  (2007), 212-254.

\bibitem{geng-shen10}
J. Geng, Z. Shen,
The Neumann problem and Helmholtz decomposition in convex domains,
\textit{J. Functional Analysis} \textbf{259} (2010), 2147-2164.

\bibitem{gu-ubachs}
Z. Gu, W. Ubachs.
A systematic study of Rayleigh-Brillouin scattering in air, N\(_2\), and O\(_2\) gases.
\textit{The Journal of chemical physics} \textbf{141} :10 (2014), 104320.

\bibitem{imanu}
M.D. Gunzburger, O.Y. Imanuvilov,
Optimal control of stationary, low Mach number, highly nonisothermal, viscous flows,
\textit{ESAIM Control Optim. Calc. Var.} \textbf{5} (2000), 477-500.
 
\bibitem{lae}
A. Laesecke,  R. Krauss, K. Stephan, W. Wagner,
 Transport Properties of Fluid Oxygen, \textit{J. Phys. Chem. Ref. Data} \textbf{19} :5 (1990), 1089-1122. 

\bibitem{lions}
 J.L. Lions,
 \textit{Quelques m\'ethodes de r\'esolution des probl\`emes aux limites non lin\'eaires}.
 Dunod et Gauthier-Villars, Paris 1969.

\bibitem{plions}
P.-L. Lions,
 \textit{Mathematical Topics in Fluid Mechanics, Vol. 2. Compressible Models}.
Lecture Series in Mathematics and its Applications. Clarendon Press, Oxford 1998.

 \bibitem{kmn}
 K. Kadoya, N. Matsunaga, A. Nagashima,
 Viscosity and thermal conductivity of dry air in the gaseous phase,
 \textit{J. Phys. Chem. Ref. Data} \textbf{14} :4 (1985), 947-969.
 
\bibitem{mitrea}
D. Mitrea,
Sharp \(L^p-\)Hodge decompositions for Lipschitz domains in \(\mathbb{R}^2\),
\textit{Adv. Differential Equations} \textbf{7} :3 (2002), 343-364.

\bibitem{mucha}
P.B. Mucha, M.Pokorn\'y,
Weak solutions to equations of steady compressible heat conducting fluids,
\textit{Math. Mod. Meth. Appl. Sci.} \textbf{20} :5 (2010), 785-813.

\bibitem{padula}
M.-R. Padula,
Uniqueness theorems for steady, compressible, heat-conducting fluids: bounded domains,
\textit{Atti Accad. Naz. Lincei Cl. Sci. Fis. Mat. Natur. Rend. Lincei (8)} \textbf{74} :6 (1983), 380-387.

\bibitem{plotni}
P.I. Plotnikov, E.V. Ruban, J. Sokolowski, 
Inhomogeneous boundary value problems for compressible Navier-Stokes and transport equations,
\textit{J. Math. Pures Appl.} \textbf{92} :2 (2009), 113-162.

\bibitem{plotni-w}
P.I. Plotnikov, W. Weigant,
Steady 3D viscous compressible flows with adiabatic exponent \(\gamma\in (1,\infty)\),
\textit{J. Math. Pures Appl.} \textbf{104} (2015), 58-82.

 \bibitem{radz}
 M. Radzina, V. Cantisani, M. Rauda, M.B. Nielsen, C. Ewertsen, F. D'Ambrosio, P. Prieditis, S. Sorrenti,
 Update on the role of ultrasound guided radiofrequency ablation for thyroid nodule treatment,
\textit{Int. J. Surg.} \textbf{41} (2017), 582-593.
 
 \bibitem{stru}
M. Struwe,
\emph{Variational methods. Applications to nonlinear partial differential equations and Hamiltonian systems.}
 Springer-Verlag, Berlin-Heidelberg 1990.

 \bibitem{valli}
 A.Valli,
 On the existence of stationary solutions to compressible Navier-Stokes equations,
 \textit{Ann. Inst. Henri Poincar\'e} \textbf{4} :1 (1987), 99-113.
 
\end{thebibliography}
\end{document}